\documentclass[a4paper, 11pt]{article}
\usepackage[margin=2cm]{geometry}
\usepackage{graphicx}
\usepackage{todonotes}
 \usepackage{hyperref}

\usepackage{amsmath,amsthm, amssymb}

\newtheorem{theorem}{Theorem}
\newtheorem{corollary}[theorem]{Corollary}
\newtheorem{lemma}[theorem]{Lemma}
\newtheorem{proposition}[theorem]{Proposition}

\newtheorem{observation}[theorem]{Observation}

\theoremstyle{definition}
\newtheorem{definition}[theorem]{Definition}

\theoremstyle{remark}
\newtheorem{remark}{Remark}

\theoremstyle{action}
\newtheorem{action}{Action}

\usepackage{tikz}
\usepackage{subcaption}

\def\--{\mbox{--}}

\usepackage[normalem]{ulem}

\newcommand{\B}[1]{B$_{#1}$-VPG}

\definecolor{forestgreen}{rgb}{0.13, 0.55, 0.13}

\usepackage{algorithmic,algorithm}

\definecolor{darkseagreen}{rgb}{0.56, 0.74, 0.56}
\definecolor{forestgreen}{rgb}{0.13, 0.55, 0.13}
\definecolor{darkblue}{rgb}{0, 0, 102}
\definecolor{lightblueArrow}{rgb}{51, 51, 255}

\usepackage{amsfonts}

\usepackage[title]{appendix}
\usepackage{lineno}
\usepackage{enumitem}

\usetikzlibrary{patterns}

\usepackage{authblk}

\definecolor{lime}{HTML}{A6CE39}
\DeclareRobustCommand{\orcidicon}{
	\begin{tikzpicture}
	\draw[lime, fill=lime] (0,0) circle [radius=0.16] 
	node[white] {{\fontfamily{qag}\selectfont \tiny ID}};
	\draw[white, fill=white] (-0.0625,0.095) 
	circle [radius=0.007];
	\end{tikzpicture}
}

\foreach \x in {DR, SKP, SJ}{%
	\expandafter\xdef\csname orcid\x\endcsname{\noexpand\href{https://orcid.org/\csname orcidauthor\x\endcsname}{\noexpand\orcidicon}}
}

\usepackage[misc,geometry]{ifsym}

\title{\B0 Representation of \\ AT-free Outerplanar Graphs}

\usepackage{authblk}

\author[1]{Sparsh Jain\orcidSJ\footnote{A part of this work was done while at Indian Institute of Technology Palakkad.}}
\author[2]{Sreejith K. Pallathumadam\orcidSKP $^{(\textrm{\Letter})}$}
\author[2]{\\ Deepak~Rajendraprasad\orcidDR}
\affil[1]{College of Computing, Georgia Institute of Technology, Atlanta, Georgia, USA \\ sparsh.jain@gatech.edu}
\affil[2]{Indian Institute of Technology Palakkad, Palakkad, India \\ 111704002@smail.iitpkd.ac.in, deepak@iitpkd.ac.in}

\begin{document}
\maketitle
\pagenumbering{arabic}
\begin{abstract}
A \emph{$k$-bend path} is a non-self-intersecting polyline in the plane made of at most $k+1$ axis-parallel line segments.  
\emph{\B{k}} is the class of graphs which can be represented as intersection graphs of $k$-bend paths in the same plane.
In this paper, we show that all AT-free outerplanar graphs  are \B0, i.e., intersection graphs of horizontal and vertical line segments in the plane.  
Our proofs are constructive and give a polynomial time \B0 drawing algorithm for the class.

Following a long line of improvements, Gon\c{c}alves, Isenmann, and Pennarun [SODA 2018] showed that all planar graphs are \B1. 
Since there are planar graphs which are not \B0, characterizing \B0 graphs among planar graphs becomes interesting. 
Chaplick et al.\ [WG 2012] had shown that it is NP-complete to recognize \B{k} graphs within \B{k+1}. Hence recognizing \B0 graphs within \B1 is NP-complete in general, but the question
is open when restricted to planar graphs.
There are outerplanar graphs and AT-free planar graphs which are not \B0. 
This piqued our interest in AT-free outerplanar graphs.  
\footnote{
	A preliminary version of this work was presented at the $8^{th}$ International Conference on Algorithms and Discrete Applied Mathematics (CALDAM) 2022 \cite{jain2022b_0}. 
		In this extended version we show that a proper superclass of AT-free outerplanar graphs, which we name as \emph{linear outerplanar graphs}, are \B0. We believe
		that his new class is interesting in its own right.
}


\end{abstract}

\section{Introduction}
	
	
A \emph{$k$-bend path} is a simple path in a two-dimensional grid with at most
$k$ bends. Geometrically, they are non-self-intersecting polylines in the plane
made of at most $k+1$ axis-parallel (horizontal or vertical) line segments.
\emph{Vertex intersection graphs of Paths on a Grid (VPG)} (resp.,
\emph{\B{k}}) is the class of graphs which can be represented as intersection
graphs of (resp., $k$-bend) paths in a two-dimensional grid.  The \emph{bend
number} of a graph $G$ in VPG is the minimum $k$ for which $G$ is in \B{k}.  
One motivation to study \B{k} graphs comes from VLSI circuit design where the
paths correspond to wires in the circuit.  A natural concern in VLSI design is
to reduce the number of bends in each path (wire) in the representation. A
second motivation is that certain algorithmic tasks become easier when
restricted to \B1 or \B0 graphs (cf. \cite{mehrabi2017approximation}).

Planar graphs have received the maximum attention from the perspective
of \B{k} representations, some of which we describe in
Section~\ref{secLiterature}. Following up on a series of results and
conjectures by various authors, Gon\c{c}alves, Isenmann, and Pennarun in 2018
showed that all planar graphs are \B1 \cite{gonccalves2018planar}.  This is
tight since many simple planar graphs like $4$-wheel, $3$-sun, triangular
prism, to name a few,  are not \B0. This makes the question of characterizing
\B0 planar graphs very appealing. Characterizing \B0 outerplanar graphs will be a good step in this direction 
since some of the structures that forbid a planar graph from being \B0
are also present among outerplanar graphs.  
Outerplanar graphs were known to be \B1 \cite{CATANZARO201784} before the same was shown for planar graphs. 
Chaplick et al. \cite{chaplick2012bend} had shown that it is NP-complete to decide whether a given graph $G$ is in \B{k} even when $G$ is guaranteed to be in \B{k+1}. 
Hence recognizing \B0 graphs with in \B1 is NP-complete in general, but the question is open when restricted to planar graphs or outerplanar graphs.

This article is an outcome of our effort to characterize \B0 outerplanar graphs. One can see from the geometry that the closed neighborhood of every vertex in a \B0 graph is an interval graph \cite{golumbic2013intersectionForDiamond}. We strengthen this necessary condition (Proposition~\ref{propNecessity}) by identifying adjacent vertices which are forced to be represented by collinear segments in any \B0 drawing. But this is still not sufficient to characterize \B0 outerplanar graphs (Figure~\ref{fig:notSufficient}).  
However,  we were able to show that, if the outerplanar graph itself is AT-free, then it is \B0 (Theorem~\ref{thm:ATfreeB0}). 
We cannot extend this result to AT-free planar graphs since we have examples of AT-free planar graphs, like $4$-wheel and triangular prism, which are not \B0.

While it is relatively easier to find a \B0 drawing for biconnected AT-free outerplanar graphs, handling cutvertices turned out to be more challenging. Rather than trying to join \B0 drawings of individual blocks, we found it easier to embed the given graph as an induced subgraph of a biconnected outerpath. 

\begin{definition}[Outerpath]
\label{def:outerpath}
		An \emph{outerpath} is an outerplanar graph which admits a planar embedding whose weak dual is a path.
\end{definition}

Note that outerpaths need not be biconnected.
All the five graphs in Figure \ref{fig:bigEg}.(a) are outerpaths.

Our proof has essentially two parts with biconnected outerpaths forming the bridge between the two. 
The first part is a structural result which shows that any AT-free outerplanar graph can be realized as an induced subgraph of a biconnected outerpath. The second part is a \B0 drawing procedure for biconnected outerpaths. 
Both the parts have a potential to be generalized. Since \B0 is easily seen to be hereditary class, the result naturally extends to all induced subgraphs of biconnected outerpaths. This prompted us to name this class (Definition~\ref{def:linear}) and study it on its own merit.

\begin{definition}[Linear Outerplanar Graph]
\label{def:linear}
	An outerplanar graph is \emph{linear} if it is {isomorphic to} a subgraph of a biconnected outerpath.
\end{definition}

\begin{figure}[h]
	
	\newcommand{\clines}{black}
	\newcommand{\cconnect}{blue}
	\newcommand{\cconnecttwo}{forestgreen}
	\newcommand{\cvertex}{black}
	
	\newcommand{\colorv}{black}
	\newcommand{\colorother}{black}
	\newcommand{\candidate}{black}
	
	\centering
	\begin{subfigure}[b]{0.9\linewidth}
		
		\centering
		\begin{tikzpicture}[scale=.7]
		
		\draw[\colorother, very thick] (0,0) -- (-3,0) -- (-3,2) -- cycle;
		\draw[\colorv, very thick] (-3,2) -- (-3,0) -- (-5,0) -- cycle;

		\draw[\colorother, very thick] (0,0) -- (4,0); 
		\draw[\colorother, very thick] (4,0) -- (5,1);
		\draw[\colorother, very thick] (5,1) -- (7,1);
		
		\draw[\colorother, very thick] (4,0) -- (5,-1);
		\draw[\colorother, very thick] (5,-1) -- (7,-1);
		\draw[\colorother, very thick] (7,-1) -- (8,-1.5);	
		
		\draw[\colorother, very thick] (0,0) -- (0,1) -- (-1,2) -- (-.5,3);
		
		\draw[\colorother, very thick] (-3,2) -- (-2,2) -- (-1.5,3);
		\draw[\colorother, very thick] (-3,2) -- (-4,2) -- (-4.5,3);
		\filldraw[\colorother] (-2,2) circle (3pt);
		\filldraw[\colorother] (-1.5,3) circle (3pt);
		\filldraw[\colorother] (-4.5,3) circle (3pt);
		
		\draw[\colorother, very thick] (-5,0) -- (-6.5,-.5) -- (-8,-.5) -- (-9,0) -- (-8,1) -- (-6.5,1) -- cycle;
		\draw[\colorother, very thick] (-9,0) -- (-10,0) -- (-11, 1) -- (-8,1) -- cycle;
		\filldraw[\colorother] (-6.5,-.5) circle (3pt);
		\filldraw[\colorother] (-8,-.5) circle (3pt);
		\filldraw[\colorother] (-8,1) circle (3pt);
		\filldraw[\colorother] (-9,0) circle (3pt);
		\filldraw[\colorother] (-10,0) circle (3pt);
		\filldraw[\colorother] (-11,1) circle (3pt);
		\filldraw[\colorother] (-6.5,1) circle (3pt);
		
		\draw[\colorother, very thick] (-10,0) -- (-11,-1) -- (-12,-1) -- (-11,0) -- cycle;
		\draw[\colorother, very thick] (-10,0) -- (-9,-1) -- (-9,-2) -- (-10,-2) -- (-10.7,-1.5) -- cycle;
		\draw[\colorother, very thick] (-10,0) -- (-10,-2);
		\draw[\colorother, very thick] (-10,0) -- (-11,.5) -- (-12,.5) -- (-13,1);
		\filldraw[\colorother] (-11,-1) circle (3pt);
		\filldraw[\colorother] (-12,-1) circle (3pt);
		\filldraw[\colorother] (-11,0) circle (3pt);
		\filldraw[\colorother] (-9,-1) circle (3pt);
		\filldraw[\colorother] (-9,-2) circle (3pt);
		\filldraw[\colorother] (-10,-2) circle (3pt);
		\filldraw[\colorother] (-10.7,-1.5) circle (3pt);
		\filldraw[\colorother] (-11,.5) circle (3pt);
		\filldraw[\colorother] (-12,.5) circle (3pt);
		\filldraw[\colorother] (-13,1) circle (3pt);
		
		\draw[\colorother, very thick] (-11,1) -- (-12,1.5) -- (-11,2) -- (-10,2);
		\draw[\colorother, very thick] (-6.5,1) -- (-7,2) -- (-8,2);
		\filldraw[\colorother] (-12,1.5) circle (3pt);
		\filldraw[\colorother] (-11,2) circle (3pt);
		\filldraw[\colorother] (-10,2) circle (3pt);
		\filldraw[\colorother] (-7,2) circle (3pt);
		\filldraw[\colorother] (-8,2) circle (3pt);
		
		\draw[\colorother, very thick] (-5,0) -- (-6,-2) -- (-5,-1.5) -- (-4,-2) -- cycle;
		\draw[\colorother, very thick] (-7,-1.5) -- (-6,-2) -- (-5,-1.5) -- (-4,-2) -- (-3,-1.5);
		\filldraw[\colorother] (-6,-2) circle (3pt);
		\filldraw[\colorother] (-5,-1.5) circle (3pt);
		\filldraw[\colorother] (-4,-2) circle (3pt);
		\filldraw[\colorother] (-7,-1.5) circle (3pt);
		\filldraw[\colorother] (-3,-1.5) circle (3pt);
		
		\draw[\colorother, very thick] (-5,0) -- (-5,1) -- (-6, 2) -- (-6.5,3);
		\filldraw[\colorother] (-5,1) circle (3pt);
		\filldraw[\colorother] (-6,2) circle (3pt);
		\filldraw[\colorother] (-6.5,3) circle (3pt);
		
		\draw[\colorother, very thick] (-5,0) -- (-4.5,1.5) -- (-5.4, 3);
		\filldraw[\colorother] (-4.5,1.5) circle (3pt);
		\filldraw[\colorother] (-5.4,3) circle (3pt);

		\draw[\colorother, very thick] (0,0) -- (1,-1) -- (0,-2) -- (-1.5,-2.5) -- (-1,-1) -- cycle;
		\draw[\colorother, very thick] (0,-2) -- (0,0);
		
		\draw[\colorother, very thick] (0,0) -- (3,-1) -- (4,-2) -- (5,-2);
		
		\draw[\colorother, very thick] (0,0) -- (1,2) -- (2,2.5) -- (3,2) -- (3.5,1) -- cycle;
		\draw[\colorother, very thick] (3,2) -- (0,0);
		
		\draw[\colorother, very thick] (1,2) -- (.5,3);
		\draw[\colorother, very thick] (.5,3) -- (1.5,3.5);
		
		\draw[\colorother, very thick] (3.5,1) -- (4.5,2);
		\draw[\colorother, very thick] (4.5,2) -- (5.5,2);
		\draw[\colorother, very thick] (5.5,2) -- (6.5,3);
		
		
		\filldraw[\colorother] (0,0) circle (3pt);
		\filldraw[\colorv] (-3,0) circle (3pt);
		\filldraw[\colorv] (-3,2) circle (3pt);
		\filldraw[\colorv] (-5,0) circle (3pt);
		\filldraw[\colorother] (4,0) circle (3pt);
		\filldraw[\colorother] (5,1) circle (3pt);
		\filldraw[\colorother] (7,1) circle (3pt);
		\filldraw[\colorother] (5,-1) circle (3pt);
		\filldraw[\colorother] (7,-1) circle (3pt);
		\filldraw[\colorother] (8,-1.5) circle (3pt);
		\filldraw[\colorother] (0,1) circle (3pt);
		\filldraw[\candidate] (-1,2) circle (3pt);
		\filldraw[\colorother] (-.5,3) circle (3pt);
		\filldraw[\colorother] (-1.5,3) circle (3pt);
		
		\filldraw[\colorother] (1,-1) circle (3pt);
		\filldraw[\candidate] (0,-2) circle (3pt);
		\filldraw[\colorother] (-1,-1) circle (3pt);
		
		\filldraw[\colorother] (5,-2) circle (3pt);
		
		\filldraw[\colorv] (3,-1) circle (3pt);
		\filldraw[\colorv] (4,-2) circle (3pt);
		\filldraw[\colorv] (-1.5,-2.5) circle (3pt);

		\filldraw[\colorother] (1,2) circle (3pt);
		\filldraw[\colorother] (2,2.5) circle (3pt);
		\filldraw[\colorother] (3,2) circle (3pt);
		\filldraw[\colorother] (3.5,1) circle (3pt);
		\filldraw[\colorother] (.5,3) circle (3pt);
		\filldraw[\colorother] (1.5,3.5) circle (3pt);
		\filldraw[\colorother] (4.5,2) circle (3pt);
		\filldraw[\colorother] (5.5,2) circle (3pt);
		\filldraw[\colorother] (6.5,3) circle (3pt);
		
		\end{tikzpicture}
	\end{subfigure}
	\caption{A Linear Outerplanar Graph}
	\label{fig:linearOuterplanar}
\end{figure}
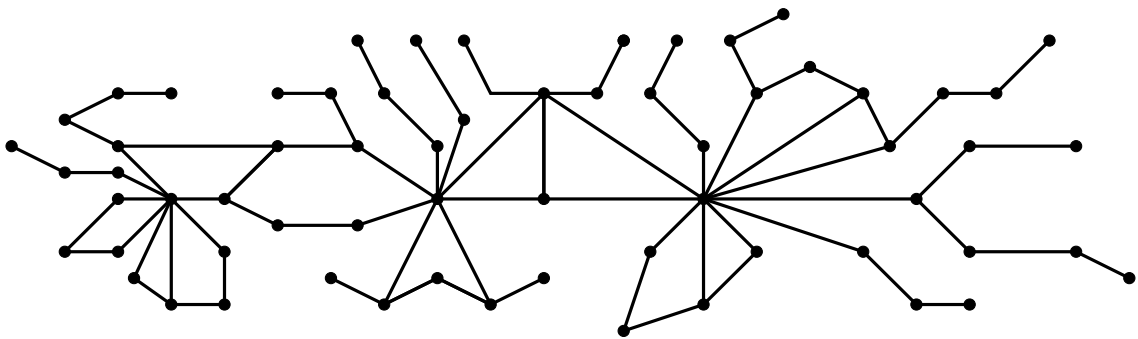

We give a complete characterization of this class in Theorem~\ref{thm:LinearCharacterization}. Though the characterization seems technical, it is very easy to visualize and gives a poly-time recognition algorithm.  
As a pleasant surprise, we also discover that every graph in this class can be realized both as an induced subgraph as well as a spanning subgraph of (different) biconnected outerpaths.
Figure~\ref{fig:linearOuterplanar} shows an example of a linear outerplanar graph.


The second part of our proof, the drawing procedure for biconnected outerpaths, can also be extended to a larger class of graphs than biconnected outerpaths, but this we set aside for a future work. 

\subsection{Organization}

After recalling some standard graph theoretic terminology in Section~\ref{secTerminology} and a brief literature review in Section~\ref{secLiterature}, we layout our proofs in three sections. 
Section \ref{sec:B0Rep} has the \B0 drawing procedure for \emph{biconnected outerpaths}. Section \ref{sec:linear} contains the characterization of linear outerplanar graphs.
In Section \ref{sec:atfree}, we prove that all AT-free outerplanar graphs are \emph{linear} thereby completing the proof of the titular result. 
We conclude with Section \ref{sec:Conclusion}, where we describe some necessary conditions for the existence of a \B0 representation. 
This may help in characterizing \B0 outerplanar graphs.

\subsection{Terminology}\label{secTerminology}

	The \emph{closed neighborhood} $N[v]$ of a vertex $v$ in a graph $G$ is the set containing $v$ and and its neighbors in $G$.
	A set of three independent vertices is called an \emph{asteroidal triple (AT)} when there exists a path among each pair of them containing no vertex from the closed neighborhood of the third vertex.
	An \emph{AT-free} graph is a graph which does not have an AT.
	
	A \emph{plane graph} is an embedding of a planar graph in the plane with no crossing edges. Let $G$ be a plane graph. 
	The \emph{dual} of $G$ is a graph that has a vertex for each face of $G$ and an edge between two of its vertices when the corresponding faces of $G$ share an edge.
	The \emph{weak dual} of $G$ is obtained from its dual by removing the vertex corresponding to the outer face. 
	An edge of $G$ incident to the outer face of $G$ is called a \emph{boundary edge} and its endpoints are called \emph{boundary neighbors} of each other.
	The remaining edges of $G$ are called \emph{internal edges}. 
	A \emph{leaf face} is a face with at most one internal edge. 
	A planar graph is \emph{outerplanar} if it has a plane embedding in which all the vertices are incident on the outer face.
	Outerplanar graphs will always be drawn in such a way that the outer face contains all the vertices, and the terminology of faces, duals, weak duals, boundary edges and internal edges will be used assuming such a plane drawing. 
	Let $G$ be an outerplanar graph. The weak dual of $G$ is a forest \cite{fleischner1974outerplanar} and we denote it by $\mathcal{T}_G$. Further, we call $\mathcal{T}_G$ a \emph{linear forest} if each component in $\mathcal{T}_G$ is a path.
	
	Let $G$ be a graph. 
	A \emph{$k^+$-vertex} in $G$ is a vertex having at least $k$ neighbors in $G$.
	A \emph{leaf edge} is an edge having one endpoint of degree one.
	A subgraph $H$ of $G$ is \emph{spanning} if $V(H)=V(G)$, and \emph{induced} if $E(H) = \{xy ~|~ x,y \in V(H),~ xy \in E(G)\}$.
	A graph induced by a subset $S$ of the vertices of $G$ is denoted by $G[S]$.
	A subset of vertices in a graph is called a \emph{separator} if its removal increases the number of components of the graph.
	A vertex $x$ is a \emph{cutvertex} if $\{x\}$ is a separator.
	A graph is \emph{k-connected} if it does not have a separator of size smaller than $k$.
	A graph is said to be \emph{connected} (resp. \emph{biconnected}) if it is 1-connected (resp. 2-connected).
	A \emph{block} of a graph is a maximal biconnected subgraph of the graph. A \emph{trivial block} is a block containing at most two vertices.



		A graph $G$ is \emph{H-free} if $G$ does not contain an induced subgraph isomorphic to $H$.
		A graph $G$ is said to be \emph{H-minor-free} if it does not contain a minor isomorphic to $H$.
		We use $C_k$ to denote the simple cycle on $k$ vertices.
		A cycle on $k$ vertices $x_0, \ldots, x_{k-1}$ where each $x_i$ is adjacent to $x_{i+1}$ (addition is modulo $k$) can also be denoted as $x_0, \ldots, x_{k-1},x_0$.
		A $C_4$ together with an additional vertex $v$ adjacent to all the vertices of $C_4$ is called a \emph{4-wheel}.
		A \emph{triangular prism} is the complement of $C_6$.
		An \emph{interval graph} is an intersection graph of a set of intervals on $\mathbb{R}$.

	\subsection{Literature}\label{secLiterature}

	The class \B{k} was introduced by Asinowski et al.\ in 2012 \cite{asinowski2012vertex}. Nevertheless, these graphs were previously studied in various forms. One of them is \emph{grid intersection
	graphs} (GIG) which are equivalent to bipartite \B0 graphs.
	The recognition problem for string graphs and hence VPG graphs is NP-complete \cite{kratochvil1991string,schaefer2003recognizing}.
	The recognition problem for 
	\B0 graphs is NP-complete  \cite{kratochvil1994intersection}.
	{\B0 characterizations are known for block graphs
		\cite{alcon2017vertex}, split graphs, chordal bull-free graphs, chordal claw-free graphs
		\cite{golumbic2013intersectionForDiamond} and cocomparability graphs \cite{pallathumadam2022characterization}.}

	\emph{Segment intersection graphs} are
	intersection graphs of line segments in the plane. Chalopin and Gon\c{c}alves
	in 2009 showed that every planar graph is a segment intersection graph
	\cite{chalopin2009every}, confirming a conjecture of Scheinerman from 1984
	\cite{scheinerman1984intersection}. One way to refine the class of segment
	intersection graphs is to restrict the number of directions permitted for the
	segments. If the number of directions is limited to two, we rediscover \B0. 
	\emph{k-DIR} graphs are intersection graphs of line segments that can lie in at most $k$ directions in the plane.
	It is known that bipartite planar graphs are $2$-DIR \cite{hartman1991grid,czyzowicz1998simple,de1991representation} and triangle-free planar graphs are $3$-DIR
	\cite{de2002triangle}. 
	West conjectured that any planar graph is $4$-DIR \cite{west1991open} which was recently refuted by Gon\c{c}alves in 2020 \cite{gonccalves2020not}. 
	Before the celebrated result by Gonçalves et al.\ that planar graphs are \B{1} \cite{gonccalves2018planar}, we had a chronology of results on \B{k} representation of  planar graphs.
	Since $2$-DIR graphs are equivalent to \B0, bipartite planar graphs \B0.
	In \cite{asinowski2012vertex}, Asinowski et al.\ showed that planar graphs are \B{3} and conjectured that it is tight. 
	Disproving this conjecture, Chaplick and Ueckerdt proved that planar graphs have a \B2 representation \cite{chaplick2012planar}. 
	This adds to the appeal for characterizing \B0 planar graphs.
	

	Outerpaths have many geometric representations like \emph{balanced circle-contact} representations \cite{alam2014balanced}, \emph{geometric simultaneous embeddings} with a matching \cite{cabello2011geometric}, and 
	\emph{partial geometric simultaneous embeddings} with another outerpath \cite{evans2014column}. 
	All these representations will extend to linear outerplanar graphs because of Theorem \ref{thm:LinearCharacterization} and Remark \ref{rem:numNewVert}. 
	Babu et al. provides an algorithm   to augment outerplanar graphs of pathwidth $p$ to biconnected outerplanar supergraphs of pathwidth $\mathcal{O}(p)$ \cite{babu20142}.
	Connectivity augmentation of outerplanar graphs using minimum number of additional edges is studied in \cite{garcia2010augmenting,kant1996augmenting}.
	Bar{\'a}t et al. have characterized the graphs with pathwidth at most two \cite{barat2012structure} and our class is a strict subclass of that. 

\section[\B0 Drawing of Biconnected Outerpaths]{\B0 Representation of Biconnected Outerpaths}
	\label{sec:B0Rep}
	
	It's drawing time! In this section, we show that every biconnected outerpath is \B0 (Theorem \ref{thm:2Cvpg}). The proof of Theorem \ref{thm:2Cvpg} is constructive and it draws a \B0 representation for any biconnected outerpath (cf. Figure \ref{fig:drawing} for example).  Since \B0 is easily seen to be
	a hereditary graph class (closed under induced subgraphs), 
	and since every linear outerplanar graph can be represented as an induced subgraph of a biconnected outerpath (Theorem \ref{thm:LinearCharacterization}), it follows that all linear outerplanar graphs are \B0.
	Furthermore, since we show in Section \ref{sec:atfree} that all AT-free outerplanar graphs are linear (Lemma~\ref{lemma:atfree}), the main result of this article follows.

	\begin{figure}
		
		\centering
		\begin{subfigure}[b]{0.4\linewidth}
			
			\centering
			\begin{tikzpicture}[scale=.45]
			\draw[black, very thick] (0,0) -- (-2,0) -- (-2,2) -- cycle;
			\draw[black, very thick] (0,0) -- (-2,-2) -- (-2,0) -- cycle;

			\node (p1) at (-3.25,.6) {\tiny $F_1$};
			\node (p1) at (9.85,.6) {\tiny $F_{13}$};
			\node (p1) at (-4.3,2) {\tiny $1$};
			\node (p1) at (-4.3,0) {\tiny $2$};
			\node (p1) at (-2.2,.6) {\tiny $3$};
			\node (p1) at (-4.3,-2.3) {\tiny $4$};
			\node (p1) at (-1.65,-2.3) {\tiny $5$};
			\node (p1) at (-0,-.5) {\tiny $6$};
			\node (p1) at (-2.25,2) {\tiny $7$};
			\node (p1) at (-1.3,4) {\tiny $8$};
			\node (p1) at (0.2,2.25) {\tiny $9$};
			\node (p1) at (2.55,-.5) {\tiny $11$};
			\node (p1) at (2.7,2.25) {\tiny $10$};
			\node (p1) at (2.55,-2.2) {\tiny $12$};
			\node (p1) at (4.27,-1.2) {\tiny $13$};
			\node (p1) at (5.6,.3) {\tiny $14$};
			\node (p1) at (5.6,1.8) {\tiny $15$};
			\node (p1) at (4.15,3.2) {\tiny $16$};
			\node (p1) at (4.55,-2.5) {\tiny $17$};
			\node (p1) at (5.6,-1.3) {\tiny $18$};
			\node (p1) at (6.85,-1.3) {\tiny $19$};
			\node (p1) at (8,-2.4) {\tiny $20$};
			\node (p1) at (8.3,-1.3) {\tiny $21$};
			\node (p1) at (9.7,-1) {\tiny $22$};
			\node (p1) at (8.9,1.4) {\tiny $23$};
			\node (p1) at (7,2.3) {\tiny $24$};
			\node (p1) at (10.5,1.3) {\tiny $25$};

			\draw [black, ultra thick] (-4,-2) -- (-4,0);
			\draw [black, ultra thick] (-4,-2) -- (-2,-2);
			
			\draw [black, ultra thick] (-2,2) -- (-1,4);
			\draw [black, ultra thick] (-1,4) -- (0,2);

			\draw[black, very thick] (-2,0) -- (-4,0) -- (-4,2) -- cycle;
			
			\draw[black, very thick] (0,0) rectangle (3,2);
			\draw[black, very thick] (3,0) -- (3,2) -- (4.5,3) -- (6,2) -- (6,0) -- cycle;
			
			\draw[black, very thick] (6,0) -- (7,2) -- (9,1) --(10,-.5) -- (8,-1) -- (6,0) -- cycle;
			\draw[black, very thick] (10,-.5) -- (9,1) -- (10.5,1) -- cycle;
			
			\draw[black, very thick] (6,0) -- (6,-1.5) -- (5,-2.5) -- (4.5,-1.5) -- cycle;		
			\draw[black, very thick] (6,0) -- (6.5,-1.5) -- (6,-1.5) -- cycle;
			
			\draw [black, ultra thick] (3,0) -- (3,-2);
			\draw [black, ultra thick] (3,-2) -- (4.5,-1.5);

			\draw [black, ultra thick] (6.5,-1.5) -- (7.5,-2.5);
			\draw [black, ultra thick] (7.5,-2.5) -- (8,-1);
			
			\filldraw[black] (0,0) circle (3pt);
			\filldraw[black] (-2,0) circle (3pt);
			\filldraw[black] (-2,2) circle (3pt);
			\filldraw[black] (-2,-2) circle (3pt);			
			
			\filldraw[black] (-4,-2) circle (3pt);
			\filldraw[black] (-4,0) circle (3pt);
			\filldraw[black] (-1,4) circle (3pt);
			\filldraw[black] (0,2) circle (3pt);
			
			\filldraw[black] (-4,2) circle (3pt);
			\filldraw[black] (3,2) circle (3pt);
			\filldraw[black] (3,0) circle (3pt);
			\filldraw[black] (4.5,3) circle (3pt);
			
			\filldraw[black] (6,2) circle (3pt);
			\filldraw[black] (6,0) circle (3pt);
			\filldraw[black] (7,2) circle (3pt);
			\filldraw[black] (9,1) circle (3pt);

			\filldraw[black] (10,-.5) circle (3pt);
			\filldraw[black] (8,-1) circle (3pt);
			\filldraw[black] (6,0) circle (3pt);
			\filldraw[black] (10.5,1) circle (3pt);
			
			\filldraw[black] (4.5,-1.5) circle (3pt);
			\filldraw[black] (6.5,-1.5) circle (3pt);
			\filldraw[black] (6,-1.5) circle (3pt);
			\filldraw[black] (5,-2.5) circle (3pt);
			
			\filldraw[black] (3,0) circle (3pt);
			\filldraw[black] (3,-2) circle (3pt);
			\filldraw[black] (6.5,-1.5) circle (3pt);
			\filldraw[black] (7.5,-2.5) circle (3pt);
			
			\filldraw[black] (8,-1) circle (3pt);
			\end{tikzpicture}
		\end{subfigure}
		\begin{subfigure}[b]{0.5\linewidth}
			\centering
			\begin{tikzpicture}[scale=.6]

			\draw[step=0.5,lightgray,ultra thin] (0,0) grid (13,10);
			
			\node (pe) at (5.8,5.2) {\tiny $1$};
			\filldraw [black] (5.9,4.9) rectangle (6.1,5.1);
			
			\draw[black, thick] (6,4) -- (6,5) -- (7,5) -- (7,4) -- (6,4) -- cycle; 
			\node (pe) at (5.8,4.2) {\tiny $2$};
			\node (pe) at (6.8,5.2) {\tiny $3$};
			\node (pe) at (6.2,3.8) {\tiny $4$};
			\node (pe) at (7.2,4.2) {\tiny $5$};
			
			\draw [black, thick, shorten <= -2pt, shorten >=-2pt] (7.1,5) -- (8.4, 5);
			\draw [black, thick, shorten <= -2pt, shorten >=-2pt] (7.1,5.1) -- (8.4, 5.1);
			\node (pe) at (7.8,5.3) {\tiny $6$};
			\node (pe) at (7.8,4.7) {\tiny $3$};
			
			\draw[black, thick] (8.5,5.1) -- (9.5,5.1) -- (9.5,3) -- (8.5,3) -- (8.5,5.1) -- cycle; 
			\node (pe) at (8.7,5.3) {\tiny $6$};
			\node (pe) at (8.3,3.2) {\tiny $7$};
			\node (pe) at (8.7,2.8) {\tiny $8$};
			\node (pe) at (9.7,3.2) {\tiny $9$};
			
			\draw[black, thick] (9.5,5.1) -- (9.5,6) -- (10.5,6) -- (10.5,5.1) -- (9.5,5.1) -- cycle; 
			\node (pe) at (10.2,4.8) {\tiny $6$};
			\node (pe) at (10.75,5.25) {\tiny $11$};
			\node (pe) at (9.3,5.8) {\tiny $9$};
			\node (pe) at (9.7,6.2) {\tiny $10$};
			
			\draw[black, thick] (10.5,6) -- (10.5,7.5) -- (11.5,7.5) -- (11.5,6) -- (10.5,6) -- cycle; 
			\node (pe) at (10.25,6.7) {\tiny $11$};
			\node (pe) at (11.3,5.8) {\tiny $10$};
			\node (pe) at (11.7,6.2) {\tiny $16$};
			\node (pe) at (11.7,7.2) {\tiny $15$};
			\node (pe) at (11.3,7.7) {\tiny $14$};
			\draw [black, thin, shorten <= -2pt, shorten >=-2pt] (11.495,6.75) -- (11.505, 6.75);
			
			\draw[black, thick] (10.5,7.5) -- (4.5,7.5) -- (4.5,8.5) -- (10.5,8.5) -- (10.5,7.5) -- cycle; 
			\node (pe) at (10.75,8.3) {\tiny $11$};
			\node (pe) at (5.25,7.3) {\tiny $14$};
			\node (pe) at (9.75,8.7) {\tiny $12$};
			\node (pe) at (4.25,8.3) {\tiny $13$};

			\draw[black, thick] (4.5,7.5) -- (4.5,1.5) -- (3.5,1.5) -- (3.5,7.5) -- (4.5,7.5) -- cycle; 
			\node (pe) at (4.7,1.75) {\tiny $13$};
			\node (pe) at (4.2,1.25) {\tiny $17$};
			\node (pe) at (4,7.7) {\tiny $14$};
			\node (pe) at (3.2,1.75) {\tiny $18$};
			
			\draw[black, thick] (3.5,7.5) -- (2.5,7.5) -- (2.5,9.5) -- (3.5,9.5) -- (3.5,7.5) -- cycle; 
			\node (pe) at (3,7.3) {\tiny $14$};
			\node (pe) at (3.7,9.3) {\tiny $19$};
			\node (pe) at (3.3,9.7) {\tiny $20$};
			\node (pe) at (2.3,9.3) {\tiny $21$};
			
			\draw[black, thick] (2.5,7.5) -- (2.5,.5) -- (1.5,.5) -- (1.5,7.5) -- (2.5,7.5) -- cycle; 
			\node (pe) at (1.7,7.75) {\tiny $14$};
			\node (pe) at (2.75,.75) {\tiny $21$};
			\node (pe) at (2.3,.3) {\tiny $22$};
			\node (pe) at (1.25,5.75) {\tiny $24$};
			\node (pe) at (1.25,2.75) {\tiny $23$};
			\draw [black, thin, shorten <= -2pt, shorten >=-2pt] (1.445,4.25) -- (1.505, 4.25);
			
			\node (pe) at (1.3,0.25) {\tiny $25$};
			\filldraw [black] (1.4,0.4) rectangle (1.6,0.6);

			\end{tikzpicture}
		\end{subfigure}
		\caption{A biconnected outerpath $G$ and a \B0 representation of it. The collinear overlapping line segments are drawn a little apart for clarity. The point segments (for eg. vertices $1$ and $25$) are drawn as black squares.} 
		\label{fig:drawing}
	\end{figure}
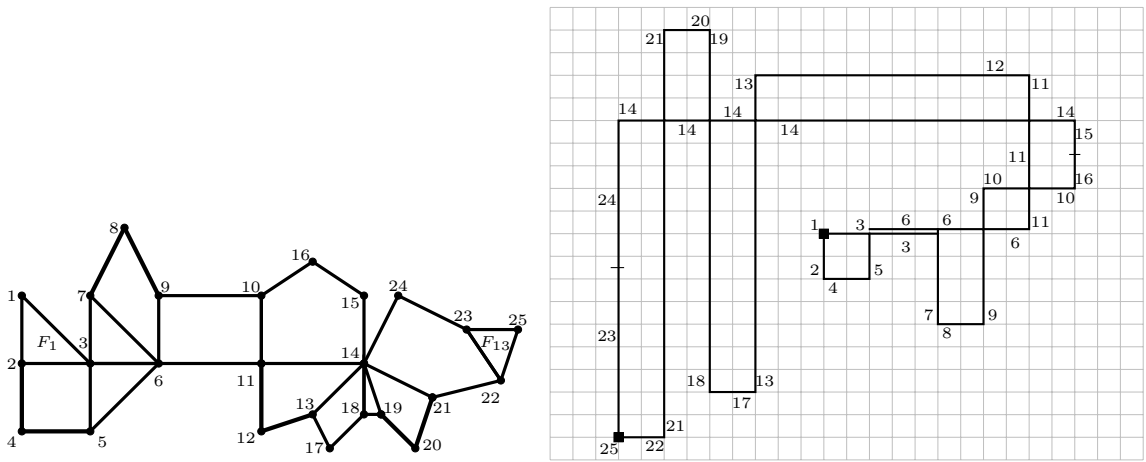
	
	\begin{theorem}
		\label{thm:2Cvpg}
		Every biconnected outerpath is \B0.
	\end{theorem}
	
	\begin{proof}
		
		Let $G$ be a biconnected outerpath with $n$ faces labeled
		$F_1, \ldots, F_n$ such that the weak dual of $G$ is the path $F_1, \ldots,
		F_n$. For each $i \in [n-1]$, the edge shared by $F_i$ and $F_{i+1}$ is
		denoted by $e_i$. For notational convenience, we set $e_n$ to be any
		boundary edge of $F_n$. For each $i \in [n]$, let $G_i$ denote the induced
		subgraph of $G$ restricted to the faces $F_1, \ldots, F_i$.
		
		In a \B0 drawing $D_i$ of $G_i$, we call a non-point horizontal (resp.,
		vertical) line segment $l$ in $D_i$ \emph{extendable} from a point $p \in
		l$ if at least one of the two infinite horizontal (resp., vertical) open
		rays starting at $p$ (but not containing $p$) does not intersect any other
		line segment of $D_i$. A point segment $l$ is said to be \emph{extendable}
		from its location $p$ if it is extendable from $p$ both as a horizontal and
		a vertical line segment.  An edge $xy$ in $G_i$ is said to be
		\emph{extendable} in $D_i$ if the line segments $l_x$ and $l_y$
		representing the vertices $x$ and $y$ are  extendable from a common point
		$p \in l_x \cap l_y$ either in the same direction or in orthogonal
		directions.  Finally a \B0 drawing $D_i$ is said to be \emph{extendable} if $e_i$ is extendable and whenever $F_i$ is a triangle, the vertex of $F_i$ not incident to $e_{i-1}$ is represented by a point segment. 
		
		If $F_1 \cong C_3$, then representing all the
		three vertices as point segments at the same point gives an extendable \B0
		drawing $D_1$ of $G_1$.
		If the length of $F_1$ is $4$ or more, then we can represent $F_1$ as the
		intersection graph of line segments laid out on the boundary of an
		axis-parallel rectangle with the endpoints of $e_1$ being orthogonal (and
		hence sharing only a corner of the rectangle). This is an extendable \B0
		drawing $D_1$ of $G_1$.    
		Let $D_i$, $i < n$, be an extendable \B0 drawing
		of $G_i$. From $D_i$, we construct an extendable \B0 drawing $D_{i+1}$ of
		$G_{i+1}$ as follows.
		
		\paragraph{Case 1 (length of $F_{i+1}$ is 4 or more).}
		Let $F_{i+1} = v_0,\ldots,v_k,v_0$, with $e_i = v_kv_0$ and $e_{i+1} =
		v_jv_{j+1}$, $j < k$. Since $D_i$ is extendable, the edge $v_kv_0$ is
		extendable in $D_i$. Extend the line segments $l_k$ and $l_0$
		(representing $v_k$ and $v_0$ respectively) in orthogonal directions to
		two points $q_k$ and $q_0$ outside of the bounding box of $D_i$. Let
		$q$ be the intersection point of the perpendiculars to $l_k$ and $l_0$
		at $q_k$ and $q_0$ respectively.  Represent the path $v_1, \ldots
		v_{k-1}$ on the two line segments from $q_0$ to $q$ and $q$ to $q_k$ such that
		$v_1$ is represented by a segment containing $q_0$, $v_{k-1}$ by a
		segment containing $q_k$ and $v_j, v_{j+1}$ by orthogonal line
		segments sharing a point. The point shared by these two segments will
		be $q_0$ when $j=0$, $q_k$ when $j=k-1$ and $q$ in all other cases.
		This gives the drawing $D_{i+1}$. 
		It is clear that the new line segments added in this stage do not intersect any other line segments in $D_i$ except $l_0$ and $l_k$. 
		It is easy to verify that the edge
		$e_{i+1} = v_jv_{j+1}$ is extendable.  Hence $D_{i+1}$ is extendable.
		
		\paragraph{Case 2 ($F_{i+1} \cong C_3$).}  
		Let $F_{i+1} = a,b,c,a$, with $e_i = ca$ and $e_{i+1} = ab$.  Since
		$D_i$ is extendable, the edge $ca$ is extendable in $D_i$ from a point
		$p$. If the line segments $l_c$ and $l_a$ are extendable in the same
		direction, then extend them to a point $q$ outside the bounding box of
		$D_i$ and represent $b$ by a point segment $l_b$ at $q$ to obtain
		$D_{i+1}$.  It is easy to check that the line segment $l_a$, the point
		segment $l_b$, and also the edge $ab$ are extendable from $q$ in
		$D_{i+1}$.  Since $ab$ is extendable and $b$ is represented by a point
		segment, $D_{i+1}$ is extendable. If $l_c$ and $l_a$ are extendable
		only in orthogonal directions, then neither of them is a point segment.
		Hence $F_i \not\cong C_3$ and hence the vertices $c$ and $a$ have no
		common neighbor in $G_i$. So the point $p$ is not contained in any line
		segment of $D_i$ other than $l_c$ and $l_a$. Represent $b$ by a point
		segment $l_b$ at $p$ to get $D_{i+1}$.
		In both the subcases, it is clear that the new line segments added in this stage do not intersect any other line segments in $D_i$ except $l_c$ and $l_a$.  
		It is easy to check that the
		line segment $l_a$, the point segment $l_b$, and also the edge $ab$ are
		extendable from $p$ in $D_{i+1}$. Since $ab$ is extendable and $b$ is
		represented by a point segment, $D_{i+1}$ is extendable.
		
		Repeating the above construction $n-1$ times gives a \B0 drawing $D_n$ of
		$G_n = G$.
	\end{proof}

	\B0 is easily seen to be a hereditary graph class (that is, closed under induced subgraphs). Thus it follows from Theorem \ref{thm:2Cvpg} that every induced subgraph of a biconnected outerpath is \B0.
	By the end of the next section, we extend this to every subgraph (not necessarily induced).
	
	\section{Characterization of Linear Outerplanar Graphs}\label{sec:linear} 
		
		In a preliminary version of the is article \cite{jain2022b_0}, we showed that 
		every AT-free outerplanar graph can be identified as an induced subgraph of a biconnected outerpath. 
		After that work, triggered by a question from Mathew C. Francis, we realized that the induced subgraphs of biconnected outerpaths can be  more exotic than AT-free outerplanar graphs, and that this class deserves to be studied on its own. 
		Our definition of \emph{linear outerplanar graphs} in \cite{jain2022b_0} was a technical choice made for the proof.
		That definition was more restrictive than the one here (Definition \ref{def:linear}).
	
	
	
	While the class of linear outerplanar graphs will inherit the rich collection of drawings and geometric representations available for biconnected outerpaths, the structure of a linear outerplanar graph is harder to describe than that of a biconnected outerpath. This section aims to do that. 
	We first build the necessary terminology for stating and proving the characterization theorem (Theorem~\ref{thm:LinearCharacterization}).
	
	Let $v$ be a cutvertex in a graph $G$. 
	For every component of $C$ of $G \setminus v$, the subgraph of $G$ induced on $V(C) \cup \{v\}$ is called a \emph{component incident to $v$}. 
	The set of components incident to $v$ is denoted by $\mathcal{C}_v$.
	A component $C$ incident to $v$ is said to be \emph{incident to a block} $B$ if $v$ is in $B$ and $C$ does not contain $B$.
	
	\begin{definition}
		Let $v$ be a cutvertex in an outerplanar graph $G$ and $C  \in \mathcal{C}_v$.
		We call $C$ \emph{small for $v$} if $C \setminus v$ is a path and \emph{big for $v$} otherwise.
		Further, when $C$ is small for $v$, we call it a \emph{tail at $v$} if $C$ (including $v$) is a path.
	\end{definition}
	\begin{figure}
		
		\newcommand{\colorv}{blue}
		\newcommand{\colorother}{black}
		\newcommand{\candidate}{blue}
		
		\centering
		\begin{subfigure}[b]{0.2\linewidth}
			
			\centering
			\begin{tikzpicture}[scale=.5]
			
			\draw[\colorother, very thick]  (-2,1) -- (-2,3) -- (2,3) -- (2,1)  -- cycle;
			
			\draw[\colorother, very thick] (-2,3) -- (0,1); 
			\draw[\colorother, very thick] (0,1) -- (2,3);
			\draw[\colorother, very thick] (0,1) -- (0,3);
			
			\draw[\colorother, very thick] (0,3) -- (0,4);
			
			\filldraw[\colorother] (0,0) circle (0pt);
			\filldraw[\colorother] (0,1) circle (3pt);
			\filldraw[\colorother] (0,3) circle (3pt);
			\filldraw[\colorother] (0,4) circle (3pt);
			\filldraw[\colorother] (-2,1) circle (3pt);
			\filldraw[\colorother] (-2,3) circle (3pt);
			\filldraw[\colorother] (2,1) circle (3pt);
			\filldraw[\colorother] (2,3) circle (3pt);
			\end{tikzpicture}
		\end{subfigure}
		\begin{subfigure}[b]{0.2\linewidth}
			
			\centering
			\begin{tikzpicture}[scale=.5]
			
			\draw[\colorother, very thick] (-.7,1) -- (-1,2) -- (0,3) -- (1,2) -- (.7,1) -- cycle;
			
			\draw[\colorother, very thick] (-.7,1) -- (-1.7,0); 
			\draw[\colorother, very thick] (.7,1) -- (1.7,0); 
			\draw[\colorother, very thick] (0,3) -- (0,4); 
			
			\draw[\colorother, very thick] (-1,2) -- (-2,3);
			\draw[\colorother, very thick] (1,2) -- (2,3);
			
			\filldraw[\colorother] (-2,3) circle (3pt);
			\filldraw[\colorother] (2,3) circle (3pt);
			
			\filldraw[\colorother] (-.7,1) circle (3pt);
			\filldraw[\colorother] (-1,2) circle (3pt);
			\filldraw[\colorother] (0,3) circle (3pt);
			\filldraw[\colorother] (.7,1) circle (3pt);
			\filldraw[\colorother] (1,2) circle (3pt);
			\filldraw[\colorother] (0,4) circle (3pt);
			\filldraw[\colorother] (1.7,0) circle (3pt);
			\filldraw[\colorother] (-1.7,0) circle (3pt);
			
			
			\end{tikzpicture}
		\end{subfigure}
		\begin{subfigure}[b]{0.2\linewidth}
			
			\centering
			\begin{tikzpicture}[scale=.5]
			
			\draw[\colorother, very thick] (0,0) -- (-1,1) -- (-1,2) -- (0,3) -- (1,2) -- (1,1) -- cycle;
			
			\draw[\colorother, very thick] (-1,1) -- (-2,0); 
			\draw[\colorother, very thick] (1,1) -- (2,0); 
			\draw[\colorother, very thick] (0,3) -- (0,4); 
			
			\filldraw[\colorother] (0,0) circle (3pt);
			\filldraw[\colorother] (-1,1) circle (3pt);
			\filldraw[\colorother] (-1,2) circle (3pt);
			\filldraw[\colorother] (0,3) circle (3pt);
			\filldraw[\colorother] (1,1) circle (3pt);
			\filldraw[\colorother] (1,2) circle (3pt);
			\filldraw[\colorother] (0,4) circle (3pt);
			\filldraw[\colorother] (2,0) circle (3pt);
			\filldraw[\colorother] (-2,0) circle (3pt);
			
			\end{tikzpicture}
		\end{subfigure}
		\begin{subfigure}[b]{0.2\linewidth}
			
			\centering
			\begin{tikzpicture}[scale=.5]
			
			\draw[\colorother, very thick] (0,0) -- (-2,0) -- (-2,2) -- (0,2) -- cycle;
			
			\draw[\colorother, very thick] (0,0) -- (1,-1); 
			\draw[\colorother, very thick] (0,0) -- (0,-1); 
			\draw[\colorother, very thick] (-2,0) -- (-3,-1); 
			\draw[\colorother, very thick] (-2,0) -- (-2,-1);
			\draw[\colorother, very thick] (-2,2) -- (-3,3);
			\draw[\colorother, very thick] (0,2) -- (0,3);
			\draw[\colorother, very thick] (0,2) -- (1,3);

			\filldraw[\colorother] (0,0) circle (3pt);
			\filldraw[\colorother] (-2,0) circle (3pt);
			\filldraw[\colorother] (-2,2) circle (3pt);
			\filldraw[\colorother] (0,2) circle (3pt);
			\filldraw[\colorother] (0,-1) circle (3pt);
			\filldraw[\colorother] (1,-1) circle (3pt);
			\filldraw[\colorother] (-2,-1) circle (3pt);
			\filldraw[\colorother] (-3,-1) circle (3pt);
			\filldraw[\colorother] (-3,3) circle (3pt);
			\filldraw[\colorother] (0,3) circle (3pt);
			\filldraw[\colorother] (1,3) circle (3pt);
			
			
			\end{tikzpicture}
		\end{subfigure}
		\begin{subfigure}[b]{0.15\linewidth}
			
			\centering
			\begin{tikzpicture}[scale=.5]
			
			\draw[\colorother, very thick] (0,0) -- (-2,0) -- (-2,2) -- (0,2) -- cycle;
			
			\draw[\colorother, very thick] (0,0) -- (1,-1); 
			\draw[\colorother, very thick] (0,0) -- (0,-1); 
			\draw[\colorother, very thick] (-2,0) -- (-3,-1); 
			\draw[\colorother, very thick] (-2,0) -- (-2,-1);
			
			\draw[\colorother, very thick] (-2,2) -- (-1,3) -- (0,2) -- cycle;

			\filldraw[\colorother] (0,0) circle (3pt);
			\filldraw[\colorother] (-2,0) circle (3pt);
			\filldraw[\colorother] (-2,2) circle (3pt);
			\filldraw[\colorother] (0,2) circle (3pt);
			\filldraw[\colorother] (0,-1) circle (3pt);
			\filldraw[\colorother] (1,-1) circle (3pt);
			\filldraw[\colorother] (-2,-1) circle (3pt);
			\filldraw[\colorother] (-3,-1) circle (3pt);
			\filldraw[\colorother] (-1,3) circle (3pt);
			
			
			\end{tikzpicture}
		\end{subfigure}
		\begin{subfigure}[b]{0.1\linewidth}
			\begin{tikzpicture}[scale=.5]
			\node (p1) at (0,0) {(a)};
			\end{tikzpicture}
		\end{subfigure}
		\begin{subfigure}[b]{0.9\linewidth}
			
			\centering
			\begin{tikzpicture}[scale=.6]
			
			\draw[\colorother, very thick] (0,0) -- (-3,0) -- (-3,2) -- cycle;
			\draw[\colorv, very thick] (-3,2) -- (-3,0) -- (-5,0) -- cycle;

			\draw[\colorother, very thick] (0,0) -- (4,0); 
			\draw[\colorother, very thick] (4,0) -- (5,1);
			\draw[\colorother, very thick] (5,1) -- (7,1);
			
			\draw[\colorother, very thick] (4,0) -- (5,-1);
			\draw[\colorother, very thick] (5,-1) -- (7,-1);
			\draw[\colorother, very thick] (7,-1) -- (8,-1.5);	
			
			\draw[\colorother, very thick] (0,0) -- (0,1);
			\draw[\colorv, very thick] (0,1) -- (-1,2);
			\draw[\colorv, very thick] (-1,2) -- (-1.5,3);
			\draw[\colorv, very thick] (-1,2) -- (-.5,3);	
			
			\draw[\colorother, very thick] (0,0) -- (1,-1) -- (0,-2) -- (-1,-1) -- cycle;	
			\draw[\colorv, very thick] (0,-2) -- (-1,-1);
			\draw[\colorv, very thick] (0,-2) -- (1,-1);
			\draw[\colorv, very thick] (0,-2) -- (0,-3.5);
			
			\draw[\colorother, very thick] (0,0) -- (3,-1);
			\draw[\colorv, very thick] (3,-1) -- (4,-2) -- (2,-2) -- cycle;
			
			\draw[\colorother, very thick] (0,0) -- (1,2) -- (2,2.5) -- (3,2) -- (3.5,1) -- cycle;
			\draw[\colorother, very thick] (3,2) -- (0,0);
			
			\draw[\colorother, very thick] (1,2) -- (.5,3);
			\draw[\colorother, very thick] (.5,3) -- (1.5,3.5);
			
			\draw[\colorother, very thick] (3.5,1) -- (4.5,2);
			\draw[\colorother, very thick] (4.5,2) -- (5.5,2);
			\draw[\colorother, very thick] (5.5,2) -- (6.5,3);
			
			\node (p1) at (0,-.5) {$\textbf{v}$};
			
			\filldraw[\colorother] (0,0) circle (3pt);
			\filldraw[\colorv] (-3,0) circle (3pt);
			\filldraw[\colorv] (-3,2) circle (3pt);
			\filldraw[\colorv] (-5,0) circle (3pt);
			\filldraw[\colorother] (4,0) circle (3pt);
			\filldraw[\colorother] (5,1) circle (3pt);
			\filldraw[\colorother] (7,1) circle (3pt);
			\filldraw[\colorother] (5,-1) circle (3pt);
			\filldraw[\colorother] (7,-1) circle (3pt);
			\filldraw[\colorother] (8,-1.5) circle (3pt);
			\filldraw[\colorother] (0,1) circle (3pt);
			\filldraw[\candidate] (-1,2) circle (5pt);
			\filldraw[\colorother] (-.5,3) circle (3pt);
			\filldraw[\colorother] (-1.5,3) circle (3pt);
			
			\filldraw[\colorother] (1,-1) circle (3pt);
			\filldraw[\candidate] (0,-2) circle (5pt);
			\filldraw[\colorother] (-1,-1) circle (3pt);
			
			\filldraw[\colorother] (0,-3.5) circle (3pt);
			
			\filldraw[\colorv] (3,-1) circle (3pt);
			\filldraw[\colorv] (4,-2) circle (3pt);
			\filldraw[\colorv] (2,-2) circle (3pt);

			\filldraw[\colorother] (1,2) circle (3pt);
			\filldraw[\colorother] (2,2.5) circle (3pt);
			\filldraw[\colorother] (3,2) circle (3pt);
			\filldraw[\colorother] (3.5,1) circle (3pt);
			\filldraw[\colorother] (.5,3) circle (3pt);
			\filldraw[\colorother] (1.5,3.5) circle (3pt);
			\filldraw[\colorother] (4.5,2) circle (3pt);
			\filldraw[\colorother] (5.5,2) circle (3pt);
			\filldraw[\colorother] (6.5,3) circle (3pt);
			
			\node (p1) at (1,-4.5) {(b)};
			\end{tikzpicture}
		\end{subfigure}
		
		\caption{Examples of outerplanar graphs which are (a) not block-safe (b) not cut-safe. In (b), $\mathcal{C}_v$ has four big components.}
		\label{fig:bigEg}
	\end{figure}
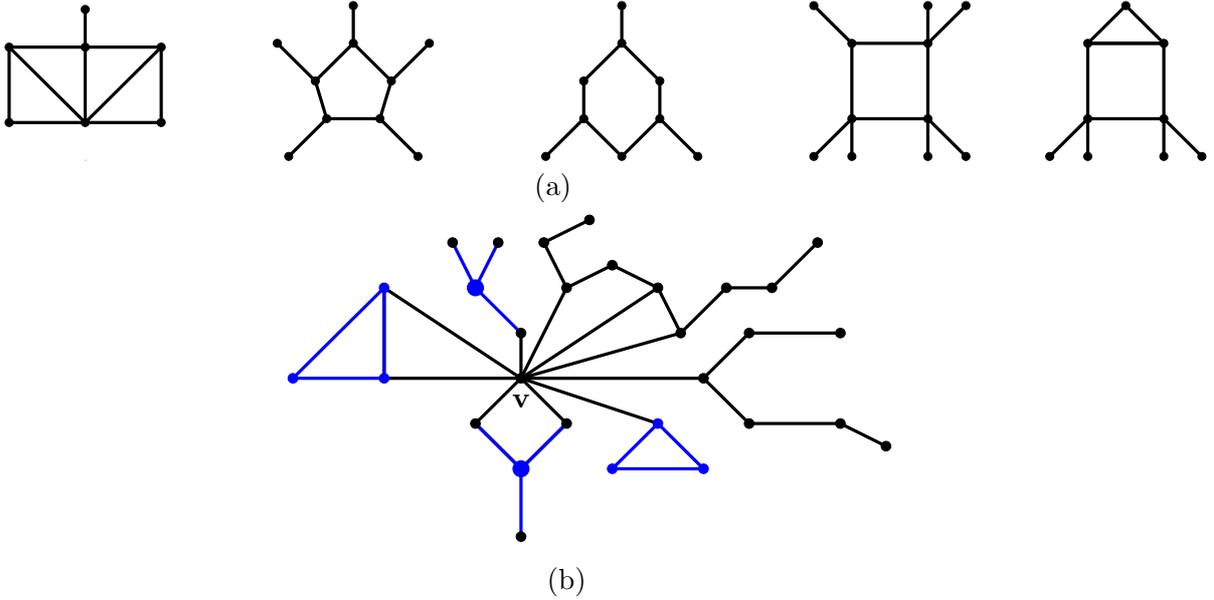
	\begin{definition}[Cut-safety]
		\label{def:cutsafe}
		A cutvertex $v$ in an outerplanar graph $G$ is said to be \emph{safe} if $\mathcal{C}_v$ contains at most two big components. The graph $G$ is said to be \emph{cut-safe} if every cutvertex in $G$ is safe.
	\end{definition}

	A set of at most two boundary edges of a block $B$ in an outerplanar graph is called \emph{antipodal} either when $B$ is a single face or when the edges belong to different leaf faces of $B$.
	
	\begin{definition}[Block-safety]
		\label{def:blocksafe}
		A nontrivial block $B$ in an outerplanar graph is called \emph{safe} if there exist two antipodal edges $a_0b_0$ and $a_1b_1$ in $B$ and the set of components incident to $B$ can be partitioned into $\mathcal{C}_0$ and $\mathcal{C}_1$ such that
		\begin{enumerate}
			\item every component in $\mathcal{C}_i$ ($i \in \{0,1\}$) is incident to either $a_i$ or $b_i$, and
			\item at most one component in $\mathcal{C}_i$ ($i \in \{0,1\}$) is incident to $a_i$ and it (if present) is a tail. 
		\end{enumerate}
		
		An outerplanar graph $G$ is said to be \emph{block-safe} if every nontrivial block in $G$ is safe. 
		The edges $a_0b_0$ and $a_1b_1$ are called \emph{terminal edges} of $B$ in $G$. A terminal edge is denoted as an ordered pair $(x,y)$ where $x=a_i$ and $y=b_i$.
		The components in $\mathcal{C}_i$ are said to be \emph{associated to} the terminal edge $(a_i,b_i)$.
	\end{definition}

	\begin{theorem}[Characterization]\label{thm:LinearCharacterization}
		An outerplanar graph $G$ is linear if and only if $G$ is cut-safe, block-safe and the weak dual of $G$ is a linear forest. Moreover, if $G$ is linear, then it can be realized both as an induced subgraph and as a spanning subgraph of (different) biconnected outerpaths.
	\end{theorem}
	
	\begin{remark}
			Note that the class of linear outerplanar graphs and outerpaths are incomparable. There are outerpaths which are not linear (cf. Figure \ref{fig:bigEg}.(a)) and linear outerplanar graphs which are not outerpaths (cf. Figure \ref{fig:2connect}.(a)).
	\end{remark} 
	
	\subsection{Proof of Theorem~\ref{thm:LinearCharacterization} (Necessity)}

	We first prove that if an outerplanar graph $G$ is linear, then $G$ is cut-safe, block-safe and $\mathcal{T}_G$ is a linear forest. Our strategy is to look at the edges of $G$ which are forced to be internal edges in any biconnected outerplanar supergraph of $G$ and then use the fact that the internal edges in a biconnected outerpath have a natural linear order. 
	It is easy to see that every internal edge of $G$, at least one edge in each face of $G$ (unless $G$ itself is a cycle), and all but at most two edges incident to any vertex of $G$ will all be internal edges in any biconnected outerplanar supergraph of $G$. 
	We can say a bit more about edges incident to a cutvertex in a nontrivial block (Lemma~\ref{lem:internalEdgeinFace}) using the simple observation below. 
	The extension of this simple observation to biconnected outerpaths  (Observation~\ref{obs:threeInternalEdges}) is a key to rest of this section.
	
	\begin{observation}\label{obs:internalEdgeRemoval}
		If $uv$ is an internal edge in a biconnected outerplanar graph $G$, then $G \setminus \{u,v\}$ has exactly two components. Moreover, both these components contain exactly one boundary neighbor each of $u$ and $v$.
	\end{observation}
	
	\begin{lemma}\label{lem:internalEdgeinFace}
		If an outerplanar graph $G$ is an induced subgraph of a biconnected outerplanar graph $G^\prime$ then for any cutvertex $v$ in a nontrivial block $B$ of $G$, one of the two boundary edges incident to $v$ in $B$ is an internal edge in $G^\prime$.
	\end{lemma}
	\begin{proof}
		Let $uv$ and $vw$ be the two boundary edges of $B$ incident to $v$ and let $x$ be a neighbor of $v$ outside $B$ in $G$.
		If both $uv$ and $vw$ are boundary edges in $G^\prime$, then $vx$ is an internal edge in $G^\prime$. But $u$ and $w$ are in the same component of $G \setminus \{v, x\}$ and hence of its supergraph $G^\prime \setminus \{v, x\}$. This contradicts Observation~\ref{obs:internalEdgeRemoval}.
	\end{proof} 
	
	For any three subsets $X, Y, Z$ of the vertices of a graph, $X$ \emph{separates} $Y$ and $Z$ if every path between $Y$ and $Z$ contains a vertex from $X$. 
	
	\begin{observation}\label{obs:threeInternalEdges}
		Let $u_0v_0, u_1v_1, u_2v_2$ be three distinct (but not necessarily disjoint) internal edges of a biconnected outerpath $G$.
		Then the endpoints of one of them, say $\{u_i, v_i\}$, separate $\{u_j, v_j\}$ from $\{u_k, v_k\}$ in $G$, where $\{i, j, k\} = \{0, 1,2\}$.
	\end{observation}
	
	\begin{lemma}\label{lemmaNecessityLinearForest}
		If an outerplanar graph $G$ is linear, then $\mathcal{T}_G$ is a linear forest. 
	\end{lemma}
	\begin{proof}
		Let $G$ be a subgraph of a biconnected outerpath $G^\prime$. 
		If $\mathcal{T}_G$ is not a linear forest, 
		then $G$ has a face $f$ with at least three internal edges.
		These remain internal edges in $G^\prime$ that violate the separation property in Observation~\ref{obs:threeInternalEdges}.
	\end{proof}
	
	\begin{lemma}\label{obs:ATinNonLinear1}
		If an outerplanar graph $G$ is linear, then $G$ is cut-safe.
	\end{lemma}
	\begin{proof}
		Let $G$ be a subgraph of a biconnected outerpath $G^\prime$. 
		Suppose $G$ has a cutvertex $v$ such that $\mathcal{C}_v$ contains three big components $C_0, C_1, C_2$.
		For each $i \in \{0, 1, 2\}$, since $C_i \setminus v$ is not a path, it either contains a face or a $3^+$-vertex. In either case, $C_i \setminus v$ will contribute an internal edge $e_i$ to $G^\prime$.
		But $e_0, e_1, e_2$ violate Observation~\ref{obs:threeInternalEdges}.
	\end{proof}

	\begin{lemma}\label{obs:ATinNonLinear2}
		If an outerplanar graph $G$ is linear, then $G$ is block-safe.
	\end{lemma}	
	\begin{proof}
		Let $G^\prime$ be an edge-minimal biconnected outerpath which is a supergraph of $G$.
		If $G$ itself is a biconnected outerpath, we are done. Otherwise, picture any nontrivial block $B$ of $G$ in $G^\prime$. Let $E'$ be the set of boundary edges of $B$ which become internal edges in $G^\prime$. 
		Since $\mathcal{T}_{G^\prime}$ is a path, it is easy to see that $E'$ is antipodal. 
		By Lemma~\ref{lem:internalEdgeinFace}, every cutvertex of $B$ in $G$ is an endpoint of an edge in $E'$.
		Let $E'$ be $\{e_0\}$ if it is singleton, and $\{e_0, e_1\}$ otherwise. The proof will be complete if we can partition the set of components $\mathcal{C}_B$ incident to $B$ into $\mathcal{C}_i$, $i \in \{0,1\}$ and label the endpoints of $e_i$ as $a_i$ and $b_i$ respecting the last condition in Definition~\ref{def:blocksafe}.
		
		By Observation \ref{obs:internalEdgeRemoval}, we get exactly two components in $G^\prime \setminus V(e_i)$. 
		Let $G_i$ be the subgraph of $G^\prime$ induced by $e_i$ and the component of $G^\prime \setminus V(e_i)$ that does not contain any vertex of $B$.
		Let $\mathcal{C}_i$ denote the components in $\mathcal{C}_B$ that are captured by $G_i$ in $G^\prime$.
		Consider the leaf face $f_i$ of ${G}_i$ containing the edge $e_i$.
		Since $f_i$ is not part of $B$, at least one edge $e_i'$ of $f_i$ is missing from $G$.
		By edge-minimality of $G$, this is a boundary edge of $G^\prime$ and hence $G_i$.
		Let $P_i$ denote the (possibly trivial) path in $f_i \setminus e_i'$ between $e_i$ and $e_i'$ that does not contain an  internal edge (if any) of $G_i$. 
		We label the endpoint of $e_i$ which meets $P_i$ as $a_i$ and the other as $b_i$.
		If $P_i$ is trivial then there is no component in $\mathcal{C}_i$ incident to $a_i$.
		If $P_i$ is not trivial, at most one component in $\mathcal{C}_i$, and that too a tail which is a subpath of $P_i$, is incident to $a_i$.
		It is easy to see that this labeling satisfies Definition~\ref{def:blocksafe}.
	\end{proof}

	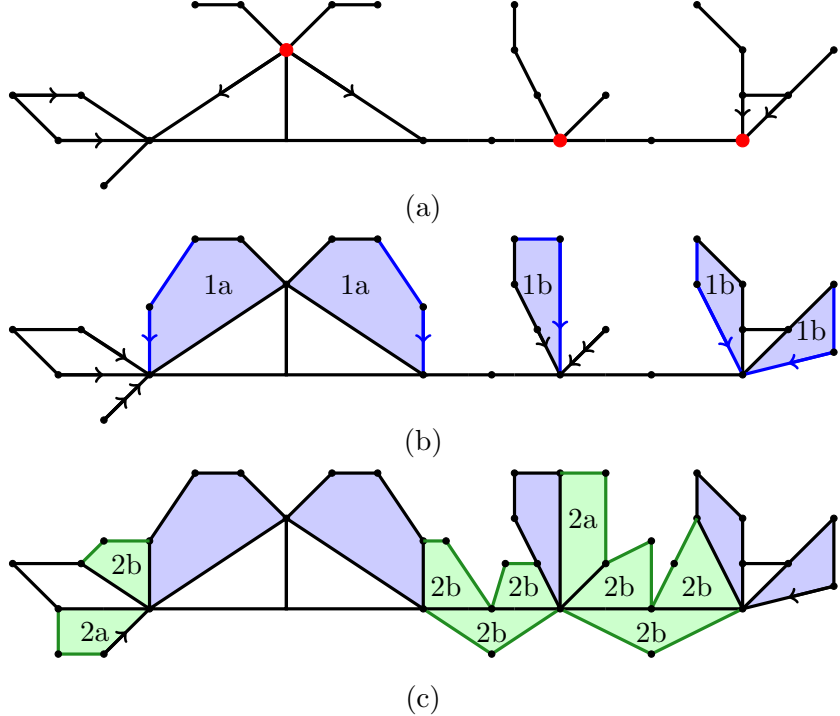
\begin{figure}[h]
		
		\newcommand{\clines}{black}
		\newcommand{\cconnect}{blue}
		\newcommand{\cconnecttwo}{forestgreen}
		\newcommand{\cvertex}{black}
		
		\centering
		\begin{subfigure}[b]{0.9\linewidth}
			
			\centering
			\begin{tikzpicture}[scale=.6,
			declare function={ edgeLine(\x,\y,\z,\zz) = \draw [black, very thick, shorten <= -2pt, shorten >=-2pt] (\x,\y) -- (\z,\zz);}]
			
			\draw[\clines, very thick] (0,0) -- (-3,0) -- (-3,2) -- cycle;
			\draw[\clines, very thick] (-3,0) -- (-6,0)  -- (-3,2) -- cycle;
			
			\draw [\clines, very thick] (-6,0) -- (-7,-1);
			
			\draw [->][\clines, very thick] (-3,2) -- (-1.5,1);
			\draw [->][\clines, very thick] (-3,2) -- (-4.5,1);
			
			\draw [->][\clines, very thick] (-8,0) -- (-7,0);
			\draw [->][\clines, very thick] (-9,1) -- (-8,1);
			
			\draw [->][\clines, very thick] (7,1) -- (7,.5);
			\draw [->][\clines, very thick] (8,1) -- (7.5,.5);

			\draw [\clines, very thick] (-6,0) -- (-7,0);
			\draw [\clines, very thick] (-7,0) -- (-8,0);
			\draw [\clines, very thick] (-8,0) -- (-9,1);
			\draw [\clines, very thick] (-9,1) -- (-7.5,1);
			\draw [\clines, very thick] (-7.5,1) -- (-6,0);
			
			\draw [\clines, very thick] (-3,2) -- (-4,3);
			\draw [\clines, very thick] (-4,3) -- (-5,3);

			\draw [\clines, very thick] (-3,2) -- (-2,3);
			\draw [\clines, very thick] (-2,3) -- (-1,3);
			
			\draw [\clines, very thick] (0,0) -- (1,0);
			\draw [\clines, very thick] (1,0) -- (2,0);
			\draw [\clines, very thick] (2,0) -- (3,0);
			
			\draw [\clines, very thick] (3,0) -- (2.5,1);
			\draw [\clines, very thick] (2.5,1) -- (2,2);
			\draw [\clines, very thick] (2,2) -- (2,3);
			
			\draw [\clines, very thick] (3,0) -- (4,1);
			
			\draw [\clines, very thick] (3,0) -- (4,0);
			\draw [\clines, very thick] (4,0) -- (5,0);
			\draw [\clines, very thick] (5,0) -- (6,0);
			\draw [\clines, very thick] (6,0) -- (7,0);
			
			\draw [\clines, very thick] (7,0) -- (7,1);
			\draw [\clines, very thick] (7,0) -- (8,1);
			\draw [\clines, very thick] (7,1) -- (8,1);
			\draw [\clines, very thick] (7,1) -- (7,2);
			\draw [\clines, very thick] (7,2) -- (6,3);
			\draw [\clines, very thick] (8,1) -- (9,2);

			\filldraw[\cvertex] (-6,0) circle (2pt);
			\filldraw[\cvertex] (-7,-1) circle (2pt);
			\filldraw[\cvertex] (-8,0) circle (2pt);
			\filldraw[\cvertex] (-9,1) circle (2pt);
			\filldraw[\cvertex] (-7.5,1) circle (2pt);
			\filldraw[red] (-3,2) circle (4pt);
			\filldraw[\cvertex] (-4,3) circle (2pt);
			\filldraw[\cvertex] (-5,3) circle (2pt);
			\filldraw[\cvertex] (-2,3) circle (2pt);
			\filldraw[\cvertex] (-1,3) circle (2pt);
			
			\filldraw[\cvertex] (0,0) circle (2pt); 
			\filldraw[\cvertex] (1.5,0) circle (2pt);
			\filldraw[red] (3,0) circle (4pt);
			\filldraw[\cvertex] (2.5,1) circle (2pt);
			\filldraw[\cvertex] (2,2) circle (2pt);
			\filldraw[\cvertex] (2,3) circle (2pt);
			\filldraw[\cvertex] (4,1) circle (2pt);
			
			\filldraw[\cvertex] (5,0) circle (2pt);
			\filldraw[red] (7,0) circle (4pt);
			\filldraw[\cvertex] (7,1) circle (2pt);
			\filldraw[\cvertex] (7,2) circle (2pt);
			\filldraw[\cvertex] (6,3) circle (2pt);
			\filldraw[\cvertex] (8,1) circle (2pt);
			\filldraw[\cvertex] (9,2) circle (2pt);
			
			\node (p1) at (0,-1.5) {(a)};
			\end{tikzpicture}
		\end{subfigure}
		\begin{subfigure}[b]{0.9\linewidth}
			\centering
			\begin{tikzpicture}[scale=.6,
			declare function={ edgeLine(\x,\y,\z,\zz) = \draw [black, very thick, shorten <= -2pt, shorten >=-2pt] (\x,\y) -- (\z,\zz);}]
			
			\draw [thin, draw=white, fill=blue, opacity=0.2]
			(-6,0) -- (-6,1.5) -- (-5,3) -- (-4,3) -- (-3,2) -- cycle;
			\draw [thin, draw=white, fill=blue, opacity=0.2]
			(0,0) -- (-3,2) -- (-2,3) -- (-1,3) -- (0,1.5) -- cycle;
			
			\draw [thin, draw=white, fill=blue, opacity=0.2]
			(3,0) -- (2,2) -- (2,3) -- (3,3) -- cycle;
			
			\draw [thin, draw=white, fill=blue, opacity=0.2]
			(7,0) -- (6,2) -- (6,3) -- (7,2) -- cycle;
			\draw [thin, draw=white, fill=blue, opacity=0.2]
			(7,0) -- (9,2) -- (9,.5) -- cycle;
			
			\draw[\clines, very thick] (0,0) -- (-3,0) -- (-3,2) -- cycle;
			\draw[\clines, very thick] (-3,0) -- (-6,0)  -- (-3,2) -- cycle;
			
			\draw [\clines, very thick] (-6,0) -- (-7,-1);
			
			\draw [\clines, very thick] (-6,0) -- (-7,0);
			\draw [\clines, very thick] (-7,0) -- (-8,0);
			\draw [\clines, very thick] (-8,0) -- (-9,1);
			\draw [\clines, very thick] (-9,1) -- (-7.5,1);
			\draw [\clines, very thick] (-7.5,1) -- (-6,0);
			
			\draw [\clines, very thick] (-3,2) -- (-4,3);
			\draw [\clines, very thick] (-4,3) -- (-5,3);

			\draw [\clines, very thick] (-3,2) -- (-2,3);
			\draw [\clines, very thick] (-2,3) -- (-1,3);
			
			\draw [\clines, very thick] (0,0) -- (1,0);
			\draw [\clines, very thick] (1,0) -- (2,0);
			\draw [\clines, very thick] (2,0) -- (3,0);
			
			\draw [\clines, very thick] (3,0) -- (2.5,1);
			\draw [\clines, very thick] (2.5,1) -- (2,2);
			\draw [\clines, very thick] (2,2) -- (2,3);
			
			\draw [\clines, very thick] (3,0) -- (4,1);
			
			\draw [\clines, very thick] (3,0) -- (4,0);
			\draw [\clines, very thick] (4,0) -- (5,0);
			\draw [\clines, very thick] (5,0) -- (6,0);
			\draw [\clines, very thick] (6,0) -- (7,0);
			
			\draw [\clines, very thick] (7,0) -- (7,1);
			\draw [\clines, very thick] (7,0) -- (8,1);
			\draw [\clines, very thick] (7,1) -- (8,1);
			\draw [\clines, very thick] (7,1) -- (7,2);
			\draw [\clines, very thick] (7,2) -- (6,3);
			\draw [\clines, very thick] (8,1) -- (9,2);
			
			\draw [->][\cconnect, very thick] (-6,1.5) -- (-6,.7);
			\draw [->][\cconnect, very thick] (0,1.5) -- (0,.7);
			
			\draw [->][\clines, very thick] (-8,0) -- (-7,0);
			\draw [->][\clines, very thick] (-7.5,1) -- (-6.5,.35);
			
			\draw [->][\clines, very thick] (-7,-1) -- (-6.5,-.5);
			\draw [->][\clines, very thick] (-7,-1) -- (-6.2,-.2);

			
			
			\draw [->][\clines, very thick] (2.5,1) -- (2.7,.6);
			\draw [->][\cconnect, very thick] (3,3) -- (3,1);
			
			\draw [->][\clines, very thick] (4,1) -- (3.2,.2);
			\draw [->][\clines, very thick] (4,1) -- (3.5,.5);
			
			
			
			\draw [->][\cconnect, very thick] (6,2) -- (6.7,.6);
			\draw [->][\cconnect, very thick] (9,0.5) -- (8,.25);
			
			\draw [\cconnect, very thick] (-5,3) -- (-6,1.5);
			\draw [\cconnect, very thick] (-6,1.5) -- (-6,0);
			\filldraw[\cvertex] (-6,1.5) circle (2pt); 
			\node (p1) at (-4.5,2) {1a};
			
			\draw [\cconnect, very thick] (-1,3) -- (0,1.5);
			\draw [\cconnect, very thick] (0,1.5) -- (0,0);
			\filldraw[\cvertex] (0,1.5) circle (2pt); 			
			\node (p1) at (-1.5,2) {1a};
			
			\draw [\cconnect, very thick] (2,3) -- (3,3);
			\draw [\cconnect, very thick] (3,3) -- (3,0);
			\filldraw[\cvertex] (3,3) circle (2pt); 
			\node (p1) at (2.5,2) {1b};
			
			\draw [\cconnect, very thick] (6,3) -- (6,2);
			\draw [\cconnect, very thick] (6,2) -- (7,0);
			\filldraw[\cvertex] (6,2) circle (2pt); 
			\node (p1) at (6.5,2) {1b};
			
			\draw [\cconnect, very thick] (9,2) -- (9,.5);
			\draw [\cconnect, very thick] (9,.5) -- (7,0);
			\filldraw[\cvertex] (9,.5) circle (2pt); 
			\node (p1) at (8.5,1) {1b};

			\filldraw[\cvertex] (-6,0) circle (2pt);
			\filldraw[\cvertex] (-7,-1) circle (2pt);
			\filldraw[\cvertex] (-8,0) circle (2pt);
			\filldraw[\cvertex] (-9,1) circle (2pt);
			\filldraw[\cvertex] (-7.5,1) circle (2pt);
			\filldraw[\cvertex] (-3,2) circle (2pt);
			\filldraw[\cvertex] (-4,3) circle (2pt);
			\filldraw[\cvertex] (-5,3) circle (2pt);
			\filldraw[\cvertex] (-2,3) circle (2pt);
			\filldraw[\cvertex] (-1,3) circle (2pt);
			
			\filldraw[\cvertex] (0,0) circle (2pt); 
			\filldraw[\cvertex] (1.5,0) circle (2pt);
			\filldraw[\cvertex] (3,0) circle (2pt);
			\filldraw[\cvertex] (2.5,1) circle (2pt);
			\filldraw[\cvertex] (2,2) circle (2pt);
			\filldraw[\cvertex] (2,3) circle (2pt);
			\filldraw[\cvertex] (4,1) circle (2pt);
			
			\filldraw[\cvertex] (5,0) circle (2pt);
			\filldraw[\cvertex] (7,0) circle (2pt);
			\filldraw[\cvertex] (7,1) circle (2pt);
			\filldraw[\cvertex] (7,2) circle (2pt);
			\filldraw[\cvertex] (6,3) circle (2pt);
			\filldraw[\cvertex] (8,1) circle (2pt);
			\filldraw[\cvertex] (9,2) circle (2pt);
			
			\node (p1) at (0,-1.5) {(b)};
			\end{tikzpicture}
		\end{subfigure}
		\begin{subfigure}[b]{0.9\linewidth}
			\centering
			\begin{tikzpicture}[scale=.6,
			declare function={ edgeLine(\x,\y,\z,\zz) = \draw [black, very thick, shorten <= -2pt, shorten >=-2pt] (\x,\y) -- (\z,\zz);}]
			
			\draw [thin, draw=white, fill=blue, opacity=0.2]
			(-6,0) -- (-6,1.5) -- (-5,3) -- (-4,3) -- (-3,2) -- cycle;
			\draw [thin, draw=white, fill=blue, opacity=0.2]
			(0,0) -- (-3,2) -- (-2,3) -- (-1,3) -- (0,1.5) -- cycle;
			
			\draw [thin, draw=white, fill=blue, opacity=0.2]
			(3,0) -- (2,2) -- (2,3) -- (3,3) -- cycle;
			
			\draw [thin, draw=white, fill=blue, opacity=0.2]
			(7,0) -- (6,2) -- (6,3) -- (7,2) -- cycle;
			\draw [thin, draw=white, fill=blue, opacity=0.2]
			(7,0) -- (9,2) -- (9,.5) -- cycle;

			\draw [thin, draw=white, fill=green, opacity=0.2]
			(-7,-1) -- (-8,-1) -- (-8,0) -- (-6,0) -- cycle;
			\draw [thin, draw=white, fill=green, opacity=0.2]
			(-7.5,1) -- (-7,1.5) -- (-6,1.5) -- (-6,0) -- cycle;
			\draw [thin, draw=white, fill=green, opacity=0.2]
			(3,3) -- (4,3) -- (4,1) -- (3,0) -- cycle;
			\draw [thin, draw=white, fill=green, opacity=0.2]
			(0,1.5) -- (.5,1.5) -- (1.5,0) -- (0,0) -- cycle;
			\draw [thin, draw=white, fill=green, opacity=0.2]
			(0,0) -- (3,0) -- (1.5,-1) -- cycle;
			\draw [thin, draw=white, fill=green, opacity=0.2]
			(1.5,0) -- (1.8,1) -- (2.5,1) -- (3,0) -- cycle;
			\draw [thin, draw=white, fill=green, opacity=0.2]
			(4,1) -- (5,1.5) -- (5,0) -- (3,0) -- cycle;
			\draw [thin, draw=white, fill=green, opacity=0.2]
			(3,0) -- (7,0) -- (5,-1) -- cycle;
			\draw [thin, draw=white, fill=green, opacity=0.2]
			(5,0) -- (6,2) -- (7,0) -- cycle;
			
			\draw[\clines, very thick] (0,0) -- (-3,0) -- (-3,2) -- cycle;
			\draw[\clines, very thick] (-3,0) -- (-6,0)  -- (-3,2) -- cycle;
			
			\draw [\clines, very thick] (-6,0) -- (-7,-1);
			
			\draw [\clines, very thick] (-6,0) -- (-7,0);
			\draw [\clines, very thick] (-7,0) -- (-8,0);
			\draw [\clines, very thick] (-8,0) -- (-9,1);
			\draw [\clines, very thick] (-9,1) -- (-7.5,1);
			\draw [\clines, very thick] (-7.5,1) -- (-6,0);
			
			\draw [\clines, very thick] (-3,2) -- (-4,3);
			\draw [\clines, very thick] (-4,3) -- (-5,3);

			\draw [\clines, very thick] (-3,2) -- (-2,3);
			\draw [\clines, very thick] (-2,3) -- (-1,3);
			
			\draw [\clines, very thick] (0,0) -- (1,0);
			\draw [\clines, very thick] (1,0) -- (2,0);
			\draw [\clines, very thick] (2,0) -- (3,0);
			
			\draw [\clines, very thick] (3,0) -- (2.5,1);
			\draw [\clines, very thick] (2.5,1) -- (2,2);
			\draw [\clines, very thick] (2,2) -- (2,3);
			
			\draw [\clines, very thick] (3,0) -- (4,1);
			
			\draw [\clines, very thick] (3,0) -- (4,0);
			\draw [\clines, very thick] (4,0) -- (5,0);
			\draw [\clines, very thick] (5,0) -- (6,0);
			\draw [\clines, very thick] (6,0) -- (7,0);
			
			\draw [\clines, very thick] (7,0) -- (7,1);
			\draw [\clines, very thick] (7,0) -- (8,1);
			\draw [\clines, very thick] (7,1) -- (8,1);
			\draw [\clines, very thick] (7,1) -- (7,2);
			\draw [\clines, very thick] (7,2) -- (6,3);
			\draw [\clines, very thick] (8,1) -- (9,2);

			\draw [\clines, very thick] (-5,3) -- (-6,1.5);
			\draw [\clines, very thick] (-6,1.5) -- (-6,0);
			\filldraw[black] (-6,1.5) circle (2pt); 
			
			\draw [\clines, very thick] (-1,3) -- (0,1.5);
			\draw [\clines, very thick] (0,1.5) -- (0,0);
			\filldraw[black] (0,1.5) circle (2pt); 			
			
			\draw [\clines, very thick] (2,3) -- (3,3);
			\draw [\clines, very thick] (3,3) -- (3,0);
			\filldraw[black] (3,3) circle (2pt); 
			
			\draw [\clines, very thick] (6,3) -- (6,2);
			\draw [\clines, very thick] (6,2) -- (7,0);
			\filldraw[black] (6,2) circle (2pt); 
			
			\draw [\clines, very thick] (9,2) -- (9,.5);
			\draw [\clines, very thick] (9,.5) -- (7,0);
			\filldraw[black] (9,.5) circle (2pt); 

			\draw [->][\clines, very thick] (-7,-1) -- (-6.5,-.5);
			
			\draw [->][\clines, very thick] (9,0.5) -- (8,.25);
			
			
			\draw [\cconnecttwo, very thick] (-7,-1) -- (-8,-1);
			\draw [\cconnecttwo, very thick] (-8,-1) -- (-8,0);
			\filldraw[\cvertex] (-8,-1) circle (2pt);
			\node (p1) at (-7.2,-.5) {2a};

			\draw [\cconnecttwo, very thick] (-7.5,1) -- (-7,1.5);
			\draw [\cconnecttwo, very thick] (-7,1.5) -- (-6,1.5);
			\filldraw[\cvertex] (-7,1.5) circle (2pt);
			\node (p1) at (-6.5,1) {2b};
			
			\draw [\cconnecttwo, very thick] (3,3) -- (4,3);
			\draw [\cconnecttwo, very thick] (4,3) -- (4,1);
			\filldraw[\cvertex] (4,3) circle (2pt);
			\node (p1) at (3.5,2) {2a};
			
			\draw [\cconnecttwo, very thick] (0,1.5) -- (.5,1.5);
			\draw [\cconnecttwo, very thick] (.5,1.5) -- (1.5,0);
			\filldraw[\cvertex] (.5,1.5) circle (2pt);
			\node (p1) at (.5,.5) {2b};
			
			\draw [\cconnecttwo, very thick] (0,0) -- (1.5,-1);
			\draw [\cconnecttwo, very thick] (1.5,-1) -- (3,0);
			\filldraw[\cvertex] (1.5,-1) circle (2pt);
			\node (p1) at (1.5,-.5) {2b};
			
			\draw [\cconnecttwo, very thick] (1.5,0) -- (1.8,1);
			\draw [\cconnecttwo, very thick] (1.8,1) -- (2.5,1);
			\filldraw[\cvertex] (1.8,1) circle (2pt);
			\node (p1) at (2.2,.5) {2b};
			
			\draw [\cconnecttwo, very thick] (4,1) -- (5,1.5);
			\draw [\cconnecttwo, very thick] (5,1.5) -- (5,0);
			\filldraw[\cvertex] (5,1.5) circle (2pt);
			\node (p1) at (4.3,.5) {2b};

			\draw [\cconnecttwo, very thick] (6,2) -- (5.5,1);
			\draw [\cconnecttwo, very thick] (5.5,1) -- (5,0);
			\filldraw[\cvertex] (5.5,1) circle (2pt);
			\node (p1) at (6,.5) {2b};
			
			\draw [\cconnecttwo, very thick] (3,0) -- (5,-1);
			\draw [\cconnecttwo, very thick] (5,-1) -- (7,0);
			\filldraw[\cvertex] (5,-1) circle (2pt);
			\node (p1) at (5,-.5) {2b};

			\filldraw[\cvertex] (-6,0) circle (2pt);
			\filldraw[\cvertex] (-7,-1) circle (2pt);
			\filldraw[\cvertex] (-8,0) circle (2pt);
			\filldraw[\cvertex] (-9,1) circle (2pt);
			\filldraw[\cvertex] (-7.5,1) circle (2pt);
			\filldraw[\cvertex] (-3,2) circle (2pt);
			\filldraw[\cvertex] (-4,3) circle (2pt);
			\filldraw[\cvertex] (-5,3) circle (2pt);
			\filldraw[\cvertex] (-2,3) circle (2pt);
			\filldraw[\cvertex] (-1,3) circle (2pt);
			
			\filldraw[\cvertex] (0,0) circle (2pt); 
			\filldraw[\cvertex] (1.5,0) circle (2pt);
			\filldraw[\cvertex] (3,0) circle (2pt);
			\filldraw[\cvertex] (2.5,1) circle (2pt);
			\filldraw[\cvertex] (2,2) circle (2pt);
			\filldraw[\cvertex] (2,3) circle (2pt);
			\filldraw[\cvertex] (4,1) circle (2pt);
			
			\filldraw[\cvertex] (5,0) circle (2pt);
			\filldraw[\cvertex] (7,0) circle (2pt);
			\filldraw[\cvertex] (7,1) circle (2pt);
			\filldraw[\cvertex] (7,2) circle (2pt);
			\filldraw[\cvertex] (6,3) circle (2pt);
			\filldraw[\cvertex] (8,1) circle (2pt);
			\filldraw[\cvertex] (9,2) circle (2pt);
			
			\node (p1) at (0,-2) {(c)};
			\end{tikzpicture}
		\end{subfigure}
		
		\caption{
			(a)~A block-safe and cut-safe outerplanar graph $G$ such that $\mathcal{T}_G$ is a linear forest. Terminal edges are shown oriented and cutvertices for Action \ref{action:tuckTail} are highlighted. 
			(b)~Intermediate graph {$G_3$}. Faces created by the two cases of Action \ref{action:tuckTail} are marked 1a and 1b respectively. 
			(c)~Final biconnected outerpath $G^\prime$. Faces created by Action \ref{action:siblingBond} are marked 2a if both the components merged are small and 2b otherwise. Since all the connectors are two-length paths, $G$ is an induced subgraph of $G^\prime$.}
		\label{fig:2connect}
	\end{figure}

	\subsection{Proof of Theorem \ref{thm:LinearCharacterization} (Sufficiency)}
	
	Now we prove the other direction of Theorem~\ref{thm:LinearCharacterization}. Let $G$ be a cut-safe and block-safe outerplanar graph such that $\mathcal{T}_G$ is a linear forest. We will construct a biconnected outerpath $G'$ which contains $G$ as a subgraph. A \emph{connector} between two nonadjacent boundary vertices $u$ and $v$ of a plane graph $H$ is either an edge $uv$ or a two-length path $(u,w,v)$ such that $w \notin V(H)$ drawn through the outer face.
	Hence the resultant graph is planar. 
	In the following construction, if every connector used is an edge (resp. two-length path), then $G$ is contained as a spanning (resp. induced) subgraph of $G'$.
	The construction of $G'$ is done in two phases. Each phase consists of repeated applications of a single  action.
	
	\begin{definition}\label{defn:maximalSmall}
		For a cutvertex $v$, a small component $C \in \mathcal{C}_v$ is called a \emph{maximal small component} if $C$ is not a subgraph of a small component incident to another cutvertex.
	\end{definition}
	
	
	\begin{action}[Tuck the tails]
		\label{action:tuckTail}
		Let $v$ be a cutvertex in $G$ and $C \in C_v$ be a maximal small component associated to a terminal edge $(a_i,b_i)$ of a nontrivial block $B$. (a) If $v = a_i$ (in which case $C$ is a tail at $a_i$), add a connector from the leaf of $C$ to $b_i$. This merges the block $B$ and the component $C$ into a new block ${{B^\prime}}$. Designate the remaining terminal edge of $B$ and the last edge $(v', b_i)$ of the connector as the two terminal edges of ${{B^\prime}}$. (b) If $v=b_i$, add a connector between $v$ and each endvertex of the path $C \setminus v$  which is not already adjacent to $v$ to form a block ${{B^\prime}}$ containing $C$ (but not $B$).
		If both the endvertices of $C\setminus v$ were adjacent to $v$, then $B^\prime = C$. 
		Designate $(u, v)$ and $(w, v)$ as the terminal edges of ${{B^\prime}}$, where $uv$ and $vw$ are the boundary edges of ${{B^\prime}}$ incident to $v$. 
		If ${{B^\prime}}$ is a single edge $uv$, designate $(u, v)$ as both the first and second terminal edge of ${{B^\prime}}$.
	\end{action}
	
	Let us call the resulting graph $G_2$. In Case (a), {the weak dual $\mathcal{T}_{B^\prime}$ of $B^\prime$} is a path formed by extending {$\mathcal{T}_B$} with a leaf edge corresponding to the dual of $a_ib_i$. The new terminal edge in ${{B^\prime}}$ is in the new leaf face and hence the two terminal edges of ${{B^\prime}}$ are antipodal. Every component incident to $B$ in $G$ (except $C$) is incident to ${{B^\prime}}$ in $G_2$. Each such component associated to a terminal edge of $B$ can be associated to the corresponding terminal edge of ${{B^\prime}}$. In Case (b), the structure of ${{B^\prime}}$ is simple since every internal edge of ${{B^\prime}}$ is incident to $v$. {It is easily verified that $\mathcal{T}_{B^\prime}$} is a path and ${{B^\prime}}$ is safe with the given terminals. Next we argue that $G_2$ is cut-safe. In Case (a), the only change to $\mathcal{C}_v$ in $G_2$ is that $C$ and the component $C_B$ containing $B$ merges into a single component $C_{{B^\prime}}$. But if $C_B$ is small for $v$ in $G$, then $C_{{B^\prime}}$ remains small for $v$ in $G_2$ and hence $v$ remains safe in $G_2$. 
	In Case (b), since ${{B^\prime}} \setminus v$ is a path, ${{B^\prime}}$ is a small component for $v$ and $v$ is safe in $G_2$. 
	{Let $u \neq v$} be a cutvertex in $G_2$ and let $C_u$ be a small component incident to $u$. Since $C$ is a maximal small component, $C_u$ does not contain $C$ and hence $C_u$ remains unaffected (and hence small) in $G_2$. So $u$ remains safe in $G_2$.
	
	We perform Action~\ref{action:tuckTail} once for each maximal small component incident to a cutvertex in $G$ to obtain a graph $G_3$. Note  that for each cutvertex $v$ in $G_3$, each small component in $\mathcal{C}_v$ is a block. 
	Also, each component incident to a nontrivial block $B$ in $G_3$ is incident at the second vertex of a terminal edge of $B$.
	For each trivial block $B = uv$ which is not a pendant edge of $G_3$, assign $(u,v)$ to be the first and $(v,u)$ to be the second terminal edge of $B$.
	Associate every component incident to $B$ at $u$ (resp. $v$) with terminal edge $(v,u)$ (resp. $(u,v)$).
	\begin{action}[Bond with your sibling]
		\label{action:siblingBond}
		Let $v$ be a cutvertex in $G_3$ and let $C_0$ and $C_1$ be two components from $\mathcal{C}_v$, chosen prioritizing small components over big ones. For $i \in \{0,1\}$, let $B_i$ be the block in $C_i$ which contains $v$, and 
		$(u_i, v)$ be a terminal edge of $B_i$.  Add a connector from $u_0$ to $u_1$. This merges $B_0$ and $B_1$ into a new block $B$. The remaining terminal edges, one each from $B_0$ and $B_1$, are designated as the terminal edges of $B$. 
	\end{action}
	
	Let us call the resulting graph $G_4$. The {weak dual $\mathcal{T}_B$ of $B$} is a path obtained by connecting two leaf vertices of {$\mathcal{T}_{B_0}$ and $\mathcal{T}_{B_1}$} through the dual vertex corresponding to the new bounded face. The new terminal edges of $B$ are antipodal in $B$. 
	Every cutvertex other than $v$ in $B_i$ ($i \in \{0,1\}$) is contained in the second terminal edge of $B_i$ which continues to be a terminal edge in $B$.
	If $C_i$ is small for an $i \in \{0,1\}$, then $B_i = C_i$ and the second terminal edge of $B_i$ has $v$ as its second vertex. 
	If both $C_0$ and $C_1$ are big for $v$, since $v$ is safe in $G_3$,  there are no other components in $\mathcal{C}_v$. Hence $v$ is no longer a cutvertex in $G_4$. Hence all the cutvertices of $B$ are contained in its terminal edges and $B$ is safe. Now we argue that $G_4$ is cut-safe. The difference from $\mathcal{C}_v$ in $G_3$ to $\mathcal{C}_v$ in $G_4$ is that $C_0$ and $C_1$ gets merged into a single component. 
	If both $C_0$ and $C_1$ are small components, then one can easily check that the new component is $B$ itself and it is small for $v$. Hence the number of big components in $\mathcal{C}_v$ does not increase and hence $v$ remains safe in $G_4$ or ceases to be a cutvertex. Let $u \neq v$ be a cutvertex in $G_4$ and let $C_u$ be a small component incident to $u$. Since $C_u$ is a block, it does not contain the cutvertex $v$ and hence remains unaffected (and thus small) in $G_4$. So $u$ remains safe in $G_4$.
	
	We repeat Action~\ref{action:siblingBond} as long as there is a cutvertex. Let $G'$ denote the resulting biconnected graph. Since each action preserves cut-safety, block-safety and linearity of the weak dual, $G'$ is a biconnected outerpath. This completes the proof of Theorem~\ref{thm:LinearCharacterization}.
	
	\begin{remark}\label{rem:numNewVert}
		From the above construction, one can check that each connector added reduces the number of blocks at least by one. 
		Hence the total number of new vertices added is less than the number of blocks in the original graph $G$.
	\end{remark}
	
	\begin{remark}\label{rem:polytime}
		Checking the cut-safety, block-safety and the linearity of weak dual of an outerplanar graph can be done in polynomial time. It can also be verified that the construction of $G'$ from $G$ can be done in polynomial time.
	\end{remark}
	
	By Theorem \ref{thm:LinearCharacterization}, any linear outerplanar graph  can be realized as an induced subgraph of a biconnected outerpath.
	This together with Theorem \ref{thm:2Cvpg} gives the following corollary.
	
	\begin{corollary}\label{cor:linearB0}
		Every linear outerplanar graph is \B0.
	\end{corollary}

	\section{Linearity of AT-free outerplanar graphs}\label{sec:atfree}
	
		By Corollary \ref{cor:linearB0}, in order to prove that AT-free outerplanar graphs are \B0, it is enough to show that they are linear. Lemma \ref{lemma:atfree} in this section asserts the same.
		Note that one may suspect whether all AT-free outerplanar graphs can be realized as induced subgraphs of biconnected AT-free outerplanar graphs. But this is incorrect. 
		For example, let $G$ be a $C_5$ together with a pendant vertex. While $G$ is AT-free outerplanar, it is easy to see that any biconnected outerplanar graph $G'$, containing $G$ as an induced subgraph, is not AT-free.
	
		The following observation is helpful in proving Lemma \ref{lemma:atfree}.

	\begin{observation}
		\label{obsFaceDeg2}
		Let $B$ be a nontrivial block of an outerplanar graph $G$. Every leaf face $f$ of $B$ contains a vertex $u$ of degree two in $G[B]$ which is not incident to any other bounded face of $B$. 
	\end{observation}
		\noindent \emph{Justification:} A face has at least three vertices and at most two vertices of a leaf face
		can be shared by another face. Hence all the remaining vertices are
		of degree two in $G[B]$ and are not incident to any other bounded face of $B$.

	\begin{lemma}\label{lemma:atfree}
		AT-free outerplanar graphs are linear.
	\end{lemma}
	\begin{proof}
		Let $G$ be an AT-free outerplanar graph.
		If $\mathcal{T}_G$ is not a linear forest, then there exists three internal edges in one face $f$ of $G$ sharing with faces, say, $f_1,f_2,f_3$. 
		One can verify that we can choose one vertex $v_i$ ($1\leq i \leq 3$) each from $f_i \setminus f$ to form an AT. 
		
		If $G$ is not cut-safe, then there exists a cutvertex $v$ where $\mathcal{C}_v$ has more than two big components, say, $C_1,C_2,C_3$. 
		That is, $C_i\setminus \{v\}$ ($1\leq i \leq 3$) is not a path and thus it has either a $3^+$-vertex $v_i$ or a face $f_i$.
		If all the neighbors of $v_i$ are adjacent to $v$, then it is easy to see that $C_i$ has $K_{2,3}$ as subgraph.
		Similarly if all the vertices of $f_i$ are adjacent to $v$, then one can verify that $C_i$ has $K_4$ as minor.
		Outerplanar graphs do not contain $K_{2,3}$ or $K_4$ as a minor \cite{chartrand1967planar}.
		Thus in both cases, there exists a vertex $v_i^\prime \in C_i\setminus \{v\}$ nonadjacent to $v$ in $C_i$.
		It is easy to see that $\{v_1^\prime, v_2^\prime, v_3^\prime\}$ forms an AT.
		
		It remains to check whether $G$ is block-safe.
		Consider any nontrivial block $B$ of $G$.
		Since $\mathcal{T}_G$ is a linear forest, $B$ either has exactly two leaf faces $f_1,f_2$ or $B$ itself is a face $f_1$.
		In the former case, Observation \ref{obsFaceDeg2} guarantees that $G[B]$ has two vertices $u_1$ and $u_2$ of degree two in $f_1$ and $f_2$ respectively.
		For any cutvertex $v$ in $B$, we denote an arbitrary neighbor of $v$ outside $B$ as $v^\prime$.
		If there exists three cutvertices in $B$, say, $v_1,v_2,v_3$, then 
		by using the path along the outer cycle (through the boundary) of $B$, one can verify that $\{v_1^\prime, v_2^\prime, v_3^\prime\}$ is an AT.
		Thus there can be at most two cutvertices in $B$.
		If $B$ itself is a face $f_1$, we can choose arbitrary boundary edges $a_0b_0$ and $a_1b_1$ of $B$ such that $b_0$ and $b_1$ are the only cutvertices of $B$.
		These two edges are antipodal in $B$, and thus $B$ satisfies the conditions of the terminal edges as per Definition \ref{def:blocksafe}, thereby showing that $B$ is safe.
		Hence we assume in the rest of the proof that $B$ contains more than one face.
		If one of the cutvertices, say $v$, is neither incident to $f_1$ nor $f_2$, then the vertices $u_1$ and $u_2$, together with $v^\prime$ form an AT.
		Hence the cutvertices incident to $B$, if any, must lie in the leaf faces of $B$. 
		Moreover, we intend to associate each cutvertex $v$ to a different leaf face of $B$ containing $v$.
		If this is not possible, that is, both the cutvertices, say $v_1,v_2$ are not incident to one of the leaf faces, say $f_2$, then $\{v_1^\prime, v_2^\prime, u_2\}$ forms an AT.
		Thus we conclude that there exist at most two cutvertices incident to $B$, and they can be associated to different leaf faces of $B$ containing them.
		We can choose arbitrary boundary edges $a_0b_0$ and $a_1b_1$ of $B$, one from each leaf face such that $b_0$ and $b_1$ are the only cutvertices of $B$.
		Clearly those edges are antipodal and their endpoints meet the conditions of the terminal edges as per Definition \ref{def:blocksafe}. 
		Thus $B$ is safe.	
	\end{proof}

	Lemma \ref{lemma:atfree} together with Corollary \ref{cor:linearB0} establish the main result.
	
	\begin{theorem}
		\label{thm:ATfreeB0}
		Every AT-free outerplanar graph is \B0.
	\end{theorem}
	
	\section{Concluding Remarks}
	\label{sec:Conclusion}
	
	We have showed that all linear outerplanar graphs are \B0. 
	However, it is easy to see that linearity is not necessary for \B0 outerplanar graphs. For example, planar bipartite graphs, and hence outerplanar bipartite graphs are \B0 \cite{hartman1991grid}. 
	But outerplanar bipartite graphs can be far from being linear, in the sense that their weak duals can be trees with arbitrarily large degrees for internal nodes. 

	One can see from a \B0 drawing that the closed neighborhood of every vertex is an interval graph \cite{golumbic2013intersectionForDiamond}. Next, we strengthen this necessary condition by identifying adjacent vertices which are forced to be represented by collinear segments in any \B0 drawing. 
	An induced $C_4$ has essentially a unique \B0 representation \cite{asinowski2012vertex}.
	A $C_4$ together with exactly one chord is called a \emph{diamond}, and the chord is called the \emph{diamond diagonal}.
	A diamond diagonal can only be drawn as the intersection of collinear line segments in every \B0 representation \cite{golumbic2013intersectionForDiamond}.
	In any \B0 representation of an odd-cycle, an odd number of edges has to be represented as the intersection of collinear line segments.
	It is also easy to verify that in a given \B0 representation $D$ of $G$, the binary relation \emph{collinearity} on the set of line segments of $D$ is an equivalence relation. 
	We combine these observations to obtain the following.
	
	\begin{proposition}\label{propNecessity}
		If a graph $G$ is \B0, then there exists a subgraph $H$ of $G$ containing
			all diamond diagonals of $G$
		such that,
		\begin{enumerate}[label=(\alph*)]
			\item for every component $C$ of $H$, the subgraph of $G$ induced by the closed neighborhood $N[C]$ of $C$ is an interval graph, and
			\item the minor of $G$ obtained by contracting every component of $H$ in $G$ is a bipartite \B0 graph.
		\end{enumerate}
	\end{proposition}
	\begin{proof}
	Since $G$ is \B0, there exists a \B0 representation $D$ of $G$. 
	Let $H$ be the spanning subgraph of $G$ in which two vertices are adjacent if and only if the corresponding line segments in $D$ are intersecting and collinear.
	Since any diamond diagonal of $G$ is represented by the intersection of collinear line segments in $D$, their endpoints are adjacent in $H$ too.
	Thus for proving (a), it remains to show that the closed neighborhood of any component of $H$ is an interval graph.
	Let $C$ be a component of $H$ and let $D_C$ be the drawing induced by the segments of $D$ that represent the vertices in $N[C]$.
	In $D_C$, if we restrict the segments that represent the vertices in $N[C] \setminus C$ to point intervals at their intersection point with a vertex in $C$,
	we obtain an interval representation for the subgraph of $G$ induced on $N[C]$.
	To obtain the minor $G_H$ of $G$ in (b), we contract every edge of $H$.
	Each component $C$ of $H$, therefore becomes a branch set of $G_H$, and can be represented as a line segment obtained as the union of all the segments representing vertices in $C$.
	This gives a \B0 representation of $G_H$.
	It is easily verified that it is a bipartite graph since there are no collinear intersections.
	\end{proof}

	\begin{remark}
	Note that since $G_H$ is bipartite, $H$ contains an odd number of edges from each odd cycle of $G$.
	\end{remark}

	
	\begin{figure}[H]
		\centering
		\begin{subfigure}[b]{0.2\linewidth}
			\centering
			\begin{tikzpicture}[scale=.5]
			
			\draw[black, thick] (-2,0) -- (0,2) -- (2,0) -- cycle;
			\draw[blue, ultra thick] (-1,1) -- (1,1) -- (0,0) -- cycle;
			
			\filldraw[black] (0,0) circle (3pt);
			\filldraw[black] (-2,0) circle (3pt);
			\filldraw[black] (2,0) circle (3pt);
			\filldraw[black] (0,2) circle (3pt);
			\filldraw[black] (-1,1) circle (3pt);
			\filldraw[black] (1,1) circle (3pt);
			\end{tikzpicture}
			\caption{3-Sun}\label{fig:sun}
		\end{subfigure}
		\begin{subfigure}[b]{0.2\linewidth}
			
			\centering
			\begin{tikzpicture}[scale=.5]
			
			\draw[black,  thick] (0,0) -- (-1,1) -- (0,2) -- (1,1) -- cycle;
			\draw[black, thick] (-1.5,0) -- (0,0) -- (-.7,-1) -- (-2,-1)  -- cycle;
			\draw[black, thick] (0,0) -- (1.5,0) -- (2,-1)  -- (.7,-1) --  cycle;
			
			\draw[black, thick] (0,0) -- (0,3);
			\draw[black, thick] (0,0) -- (-3,-1.5);
			\draw[black, thick] (0,0) -- (3,-1.5);

			\draw[blue, ultra thick] (0,0) -- (0,2);
			\draw[blue, ultra thick] (0,0) -- (-2,-1);
			\draw[blue, ultra thick] (0,0) -- (2,-1);
			
			\filldraw[black] (0,0) circle (3pt);
			\filldraw[black] (-1,1) circle (3pt);
			\filldraw[black] (0,2) circle (3pt);
			\filldraw[black] (1,1) circle (3pt);
			\filldraw[black] (-1.5,0) circle (3pt);
			\filldraw[black] (-.7,-1) circle (3pt);
			\filldraw[black] (-2,-1) circle (3pt);
			\filldraw[black] (1.5,0) circle (3pt);
			\filldraw[black] (2,-1) circle (3pt);
			\filldraw[black] (.7,-1) circle (3pt);
			\filldraw[black] (0,3) circle (3pt);
			\filldraw[black] (-3,-1.5) circle (3pt);
			\filldraw[black] (3,-1.5) circle (3pt);

			\end{tikzpicture}
			\caption{Kite}
		\end{subfigure}
		\begin{subfigure}[b]{0.25\linewidth}
			
			\centering
			\begin{tikzpicture}[scale=.5]
			
			\draw[red, thick] (0,0) -- (-2,0) -- (-2,2) -- (0,2) -- cycle;
			
			\draw[black, thick] (-2,0) -- (-1,-1) -- (-2,-2) -- cycle;
			\draw[black, thick] (0,0) -- (0,-2) -- (-1,-1) -- cycle;
			
			\draw[blue, ultra thick] (-2,0) -- (-1,-1);
			\draw[blue, ultra thick] (0,0) -- (-1,-1);
			
			\draw[blue, ultra thick] (-2,0) -- (0,0);
			
			\filldraw[black] (0,0) circle (3pt);
			\filldraw[black] (-2,0) circle (3pt);
			\filldraw[black] (-2,2) circle (3pt);
			\filldraw[black] (0,2) circle (3pt);
			\filldraw[black] (-2,-2) circle (3pt);
			\filldraw[black] (0,-2) circle (3pt);
			\filldraw[black] (-1,-1) circle (3pt);
			
			\end{tikzpicture}
			\caption{Bookmark}
		\end{subfigure}
		\begin{subfigure}[b]{0.3\linewidth}
			\centering
			\begin{tikzpicture}[scale=.5]
			
			\draw[black, thick] (0,0) -- (-1.5,1) -- (-1,2.5) -- (1,2.5) -- (1.5,1) -- cycle;
			
			\draw[black, thick] (0,0) -- (1.5,1) -- (3,1) -- (1.5,0) -- cycle;
			\draw[blue, ultra thick] (1.5,1) -- (1.5,0);
			\draw[black, thick] (1.5,0) -- (3,0);
			
			\draw[black, thick] (0,0) -- (0,-1) -- (-1.5,0) -- (-1.5,1) -- cycle;
			\draw[blue, ultra thick] (0,0) -- (-1.5,0);
			\draw[black, thick] (-1.5,0) -- (-1.5,-1);
			
			\draw[black, thick] (-1.5,1) -- (-2.6,.5) -- (-2.2,2) -- (-1,2.5) -- cycle;
			\draw[blue, ultra thick] (-1.5,1) -- (-2.2,2);
			\draw[black, thick] (-2.2,2) -- (-3.2,1.5);
			
			\draw[black, thick] (-1,2.5) -- (-2,3.5) -- (0,3.5) -- (1,2.5) -- cycle;
			\draw[blue, ultra thick] (-1,2.5) -- (0,3.5);
			\draw[black, thick] (0,3.5) -- (-1,4.5);
			
			\draw[black, thick] (1,2.5) -- (1.7,3.5) -- (2.2,2) -- (1.5,1) -- cycle;
			\draw[blue, ultra thick] (1,2.5) -- (2.2,2);
			\draw[black, thick] (2.2,2) -- (3,3.2);
			
			\draw[forestgreen, ultra thick] (-1,2.5) -- (1,2.5);
			
			\filldraw[black] (0,0) circle (3pt);
			\filldraw[black] (-1.5,1) circle (3pt);
			\filldraw[black] (-1,2.5) circle (3pt);
			\filldraw[black] (1,2.5) circle (3pt);
			\filldraw[black] (1.5,1) circle (3pt);
			
			\filldraw[black] (3,1) circle (3pt);
			\filldraw[black] (1.5,0) circle (3pt);
			\filldraw[black] (3,0) circle (3pt);
			\filldraw[black] (0,-1) circle (3pt);
			\filldraw[black] (-1.5,0) circle (3pt);
			\filldraw[black] (-1.5,-1) circle (3pt);

			\filldraw[black] (-2.6,.5) circle (3pt);
			\filldraw[black] (-2.2,2) circle (3pt);
			\filldraw[black] (-3.2,1.5) circle (3pt);
			
			\filldraw[black] (-2,3.5) circle (3pt);
			\filldraw[black] (0,3.5) circle (3pt);
			\filldraw[black] (-1,4.5) circle (3pt);
			
			\filldraw[black] (1.7,3.5) circle (3pt);
			\filldraw[black] (2.2,2) circle (3pt);
			\filldraw[black] (3,3.2) circle (3pt);
			
			\end{tikzpicture}
			\caption{Ninja Star}\label{fig:ninja}
		\end{subfigure}
		\begin{subfigure}[b]{0.3\linewidth}
			\centering
			\begin{tikzpicture}[scale=.5]
			
			\draw[red, thick] (0,0) -- (-1.5,1) -- (-1,2.5) -- (1,2.5) -- (1.5,1) -- cycle;
			
			\draw[red, thick] (-1,2.5) -- (-1,3.5) -- (1,3.5) -- (1,2.5) -- cycle;
			\draw[red, thick] (1,2.5) -- (2,3) -- (2.5,1.5) -- (1.5,1) -- cycle;
			\draw[red, thick] (-1,2.5) -- (-2,3) -- (-2.5,1.5) -- (-1.5,1) -- cycle;
			\draw[red, thick] (1.5,1) -- (2.5,.5) -- (1,-.5) -- (0,0) -- cycle;
			\draw[red, thick] (-1.5,1) -- (-2.5,.5) -- (-1,-.5) -- (0,0) -- cycle;
			
			\draw[forestgreen, ultra thick] (-1,2.5) -- (1,2.5);
			
			\filldraw[black] (0,0) circle (3pt);
			\filldraw[black] (-1.5,1) circle (3pt);
			\filldraw[black] (-1,2.5) circle (3pt);
			\filldraw[black] (1,2.5) circle (3pt);
			\filldraw[black] (1.5,1) circle (3pt);
			\filldraw[black] (-1,3.5) circle (3pt);
			\filldraw[black] (1,3.5) circle (3pt);
			\filldraw[black] (2,3) circle (3pt);
			\filldraw[black] (2.5,1.5) circle (3pt);
			
			\filldraw[black] (-2,3) circle (3pt);
			\filldraw[black] (-2.5,1.5) circle (3pt);
			\filldraw[black] (2.5,.5) circle (3pt);
			\filldraw[black] (1,-.5) circle (3pt);
			\filldraw[black] (-2.5,.5) circle (3pt);
			\filldraw[black] (-1,-.5) circle (3pt);
			\end{tikzpicture}
			\caption{Odd cycle with $C_4$ petals}\label{fig:C4petal}
		\end{subfigure}
		\begin{subfigure}[b]{0.3\linewidth}
			\centering
			\begin{tikzpicture}[scale=.5]
			
			\draw[red, double, thick] (0,0) -- (-1,1) -- (-.75,2.5) -- (.75,2.5) -- (1,1) -- cycle;
			\draw[red, double, thick] (0,0) -- (-.75,-1.5) -- (-2,-2) -- (-3,-1) -- (-2,0) -- cycle;
			\draw[red, double, thick] (0,0) -- (.75,-1.5) -- (2,-2) -- (3,-1) -- (2,0) -- cycle;
			
			\draw[white, double, thick] (0,0) -- (1,1);
			\draw[white, double, thick] (0,0) -- (-2,0);
			\draw[white, double, thick] (0,0) -- (.75,-1.5);
			
			\draw[blue, ultra thick] (0,0) -- (1,1);
			\draw[blue, ultra thick] (0,0) -- (-2,0);
			\draw[blue, ultra thick] (0,0) -- (.75,-1.5);
			
			\filldraw[black] (0,0) circle (3pt);
			\filldraw[black] (-1,1) circle (3pt);
			\filldraw[black] (-.75,2.5) circle (3pt);
			\filldraw[black] (.75,2.5) circle (3pt);
			\filldraw[black] (1,1) circle (3pt);
			\filldraw[black] (.75,-1.5) circle (3pt);
			\filldraw[black] (-.75,-1.5) circle (3pt);
			\filldraw[black] (-2,-2) circle (3pt);
			\filldraw[black] (2,-2) circle (3pt);
			\filldraw[black] (-3,-1) circle (3pt);
			\filldraw[black] (3,-1) circle (3pt);
			\filldraw[black] (2,0) circle (3pt);
			\filldraw[black] (-2,0) circle (3pt);	
			\end{tikzpicture}
			\caption{}
		\end{subfigure}
		\begin{subfigure}[b]{0.3\linewidth}
			\centering
			\begin{tikzpicture}[scale=.6]

			\draw[red, double, thick] (0,0) -- (-.75,1) -- (-.75,2.5) -- (.75,2.5) -- (.75,1) -- cycle;
			\draw[red, double, thick] (0,0) -- (.75,1) -- (2,1.5) -- (2.5,.75) -- (1.5,0) -- cycle;
			\draw[red, double, thick] (0,0) -- (-.75,1) -- (-2,1.5) -- (-2.5,.75) -- (-1.8,.1) -- cycle;
			
			\draw[red, double, thick] (0,0) -- (1.5,0) -- (2.5,-.5) -- (2,-1.25) -- (.75,-1) -- cycle;
			
			\draw[red, double, thick] (0,0) -- (.75,-1) -- (.75,-2) -- (-.75,-2) -- (-.75,-1) -- cycle;
			
			\draw[red, double, thick] (0,0) -- (-.75,-1) -- (-1.8,-1.8) -- (-2.7,-1) -- (-1.8,-.3) -- cycle;
			
			\draw[white, double, thick] (0,0) -- (-.75,1);
			\draw[white, double, thick] (0,0) -- (.75,1);
			\draw[white, double, thick] (0,0) -- (1.5,0);
			\draw[white, double, thick] (0,0) -- (.75,-1);
			\draw[white, double, thick] (0,0) -- (-.75,-1);

			\draw[blue, ultra thick] (0,0) -- (-.75,1);
			\draw[black, thick] (0,0) -- (.75,1);
			\draw[blue, ultra thick] (0,0) -- (1.5,0);
			\draw[black, thick] (0,0) -- (.75,-1);
			\draw[blue, ultra thick] (0,0) -- (-.75,-1);
			
			\filldraw[black] (0,0) circle (3pt);
			\filldraw[black] (-.75,1) circle (3pt);
			\filldraw[black] (-.75,2.5) circle (3pt);
			\filldraw[black] (.75,2.5) circle (3pt);
			\filldraw[black] (.75,1) circle (3pt);
			
			\filldraw[black] (2,1.5) circle (3pt);
			\filldraw[black] (2.5,.75) circle (3pt);
			\filldraw[black] (1.5,0) circle (3pt);

			\filldraw[black] (-2,1.5) circle (3pt);
			\filldraw[black] (-2.5,.75) circle (3pt);
			\filldraw[black] (-1.8,.1) circle (3pt);
			
			\filldraw[black] (2.5,-.5) circle (3pt);
			\filldraw[black] (2,-1.25) circle (3pt);
			\filldraw[black] (-.75,-1) circle (3pt);
			
			\filldraw[black] (.75,-1) circle (3pt);
			
			\filldraw[black] (.75,-2) circle (3pt);
			\filldraw[black] (-.75,-2) circle (3pt);
			
			\filldraw[black] (-1.8,-1.8) circle (3pt);
			\filldraw[black] (-2.7,-1) circle (3pt);
			\filldraw[black] (-1.8,-.3) circle (3pt);

			\end{tikzpicture}
			\caption{}
		\end{subfigure}
		\caption{A few outerplanar graphs which are not \B0. The {\color{blue}blue} edges represent forced collinear edges.
		The {\color{forestgreen}green} edges in (d) and (e) represent the edges arbitrarily chosen in an odd cycle to become collinear in a \B0 drawing. 
		The 4-cycles in (e) resemble petals and hence the name.
		The {\color{red}red double} edges in (f) and (g) represent $C_4$ edges where the $C_4$ (petal) is not drawn to avoid cluttering.}
		\label{fig:forbidden}
	\end{figure}
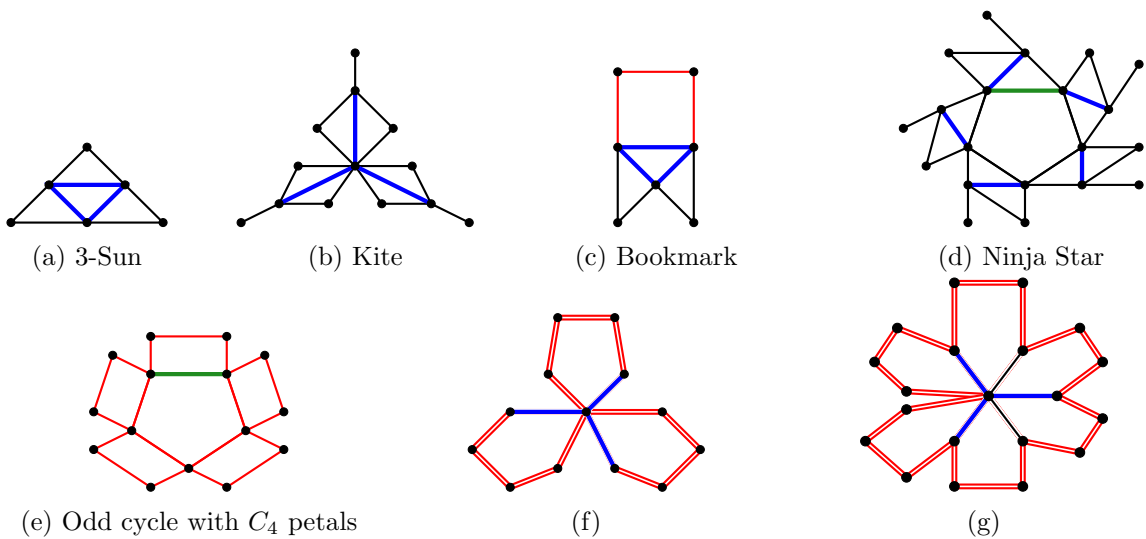
	
	Proposition~\ref{propNecessity} shows that the outerplanar graphs in Figure \ref{fig:forbidden} are not \B0.	
	Nonetheless, the above proposition is not sufficient to characterize \B0 even among outerplanar graphs. Figure \ref{fig:notSufficient} is such an example.	

	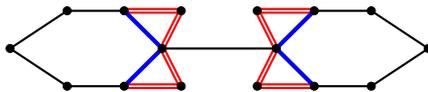
\begin{figure}[H]
		
		\centering
		\begin{subfigure}[b]{0.5\linewidth}
			\centering
			\begin{tikzpicture}[scale=.5]
			\centering
					\draw[black, thick] (0,0) -- (-3,0);
					\draw[black, thick] (0,0) -- (1,1) -- (2.5,1) -- (4,0) -- (2.5,-1) -- (1,-1) -- cycle;
					
					\draw[red, double, thick] (0,0) -- (-.5,1);
					\draw[red, double, thick] (-.5,1) -- (1,1);
					
					\draw[red, double, thick] (0,0) -- (-.5,-1);
					\draw[red, double, thick] (-.5,-1) -- (1,-1);
					
					
					\draw[blue, ultra thick] (0,0) -- (1,1);
					\draw[blue, ultra thick] (0,0) -- (1,-1);
					
					\draw[black, thick] (-3,0) -- (-4,1) -- (-5.5,1) -- (-7,0) -- (-5.5,-1) -- (-4,-1) -- cycle;
					
					\draw[red, double, thick] (-3,0) -- (-2.5,1);
					\draw[red, double, thick] (-2.5,1) -- (-4,1);
					
					
					\draw[red, double, thick] (-3,0) -- (-2.5,-1);
					\draw[red, double, thick] (-2.5,-1) -- (-4,-1);
					
					\draw[blue, ultra thick] (-3,0) -- (-4,1);
					\draw[blue, ultra thick] (-3,0) -- (-4,-1);

					\filldraw[black] (0,0) circle (3pt);
					\filldraw[black] (1,1) circle (3pt);
					\filldraw[black] (2.5,1) circle (3pt);
					\filldraw[black] (4,0) circle (3pt);
					\filldraw[black] (2.5,-1) circle (3pt);
					\filldraw[black] (1,-1) circle (3pt);

					\filldraw[black] (-3,0) circle (3pt);
					\filldraw[black] (-4,1) circle (3pt);
					\filldraw[black] (-5.5,1) circle (3pt);
					\filldraw[black] (-7,0) circle (3pt);
					\filldraw[black] (-5.5,-1) circle (3pt);
					\filldraw[black] (-4,-1) circle (3pt);
					
					\filldraw[black] (-.5,1) circle (3pt);
					\filldraw[black] (-.5,-1) circle (3pt);
					\filldraw[black] (-2.5,1) circle (3pt);
					\filldraw[black] (-2.5,-1) circle (3pt);

			\end{tikzpicture}
		\end{subfigure}
		\caption{An outerplanar graph which is not \B0. This shows that the conditions in Proposition \ref{propNecessity} was not sufficient. The {\color{red}red double} edges represent $C_4$ edges where the $C_4$ is not drawn to avoid cluttering. The {\color{blue}blue} edges are forced to be collinear. Hence the bridge has to be realized as an orthogonal intersection and this prevents a non-crossing drawing of the two $6$-cycles.}
		\label{fig:notSufficient}
	\end{figure}

	\subsubsection*{Acknowledgments}
	We thank K.~Murali Krishnan and Mathew C. Francis for fruitful discussions.

\bibliographystyle{splncs04}
\bibliography{reference}

\begin{thebibliography}{10}
\providecommand{\url}[1]{\texttt{#1}}
\providecommand{\urlprefix}{URL }
\providecommand{\doi}[1]{https://doi.org/#1}

\bibitem{alam2014balanced}
Alam, M., Eppstein, D., Goodrich, M.T., Kobourov, S.G., Pupyrev, S., et~al.:
  Balanced circle packings for planar graphs. In: International Symposium on
  Graph Drawing. pp. 125--136. Springer (2014)

\bibitem{alcon2017vertex}
Alc{\'o}n, L., Bonomo, F., Mazzoleni, M.P.: Vertex intersection graphs of paths
  on a grid: characterization within block graphs. Graphs and Combinatorics
  \textbf{33}(4),  653--664 (2017)

\bibitem{asinowski2012vertex}
Asinowski, A., Cohen, E., Golumbic, M.C., Limouzy, V., Lipshteyn, M., Stern,
  M.: Vertex intersection graphs of paths on a grid. J. Graph Algorithms Appl.
  \textbf{16}(2),  129--150 (2012)

\bibitem{babu20142}
Babu, J., Basavaraju, M., Chandran, L.S., Rajendraprasad, D.: 2-{C}onnecting
  outerplanar graphs without blowing up the pathwidth. Theoretical Computer
  Science  \textbf{554},  119--134 (2014)

\bibitem{barat2012structure}
Bar{\'a}t, J., Hajnal, P., Lin, Y., Yang, A.: On the structure of graphs with
  path-width at most two. Studia Scientiarum Mathematicarum Hungarica
  \textbf{49}(2),  211--222 (2012)

\bibitem{cabello2011geometric}
Cabello, S., van Kreveld, M.J., Liotta, G., Meijer, H., Speckmann, B., Verbeek,
  K.: Geometric simultaneous embeddings of a graph and a matching. J. Graph
  Algorithms Appl.  \textbf{15}(1),  79--96 (2011)

\bibitem{de2002triangle}
de~Castro, N., Cobos, F.J., Dana, J.C., M{\'a}rquez, A., Noy, M.:
  {T}riangle-free {P}lanar {G}raphs and {S}egment {I}ntersection {G}raphs. J.
  Graph Algorithms Appl.  \textbf{6}(1),  7--26 (2002)

\bibitem{CATANZARO201784}
Catanzaro, D., Chaplick, S., Felsner, S., Halldórsson, B.V., Halldórsson,
  M.M., Hixon, T., Stacho, J.: Max point-tolerance graphs. Discrete Applied
  Mathematics  \textbf{216},  84--97 (2017)

\bibitem{chalopin2009every}
Chalopin, J., Gon{\c{c}}alves, D.: Every planar graph is the intersection graph
  of segments in the plane. In: Proceedings of the forty-first annual ACM
  symposium on Theory of computing. pp. 631--638 (2009)

\bibitem{chaplick2012bend}
Chaplick, S., Jel{\'\i}nek, V., Kratochv{\'\i}l, J., Vysko{\v{c}}il, T.:
  Bend-bounded path intersection graphs: {S}ausages, noodles, and waffles on a
  grill. In: International Workshop on Graph-Theoretic Concepts in Computer
  Science, 274--285. Springer (2012)

\bibitem{chaplick2012planar}
Chaplick, S., Ueckerdt, T.: Planar graphs as {VPG}-graphs. In: International
  Symposium on Graph Drawing, 174--186. Springer (2012)

\bibitem{chartrand1967planar}
Chartrand, G., Harary, F.: Planar permutation graphs. In: Annales de l'IHP
  Probabilit{\'e}s et statistiques. vol.~3, pp. 433--438 (1967)

\bibitem{czyzowicz1998simple}
Czyzowicz, J., Kranakis, E., Urrutia, J.: A simple proof of the representation
  of bipartite planar graphs as the contact graphs of orthogonal straight line
  segments. Information Processing Letters  \textbf{66}(3),  125--126 (1998)

\bibitem{de1991representation}
De~Fraysseix, H., Ossona~de Mendez, P., Pach, J.: Representation of planar
  graphs by segments. Intuitive Geometry  \textbf{63},  109--117 (1991)

\bibitem{evans2014column}
Evans, W., Kusters, V., Saumell, M., Speckmann, B.: Column planarity and
  partial simultaneous geometric embedding. In: International Symposium on
  Graph Drawing. pp. 259--271. Springer (2014)

\bibitem{fleischner1974outerplanar}
Fleischner, H.J., Geller, D.P., Harary, F.: Outerplanar graphs and weak duals.
  The Journal of the Indian Mathematical Society  \textbf{38}(1-4),  215--219
  (1974)

\bibitem{garcia2010augmenting}
Garc{\'\i}a, A., Hurtado, F., Noy, M., Tejel, J.: Augmenting the {C}onnectivity
  of {O}uterplanar {G}raphs. Algorithmica  \textbf{56}(2),  160--179 (2010)

\bibitem{golumbic2013intersectionForDiamond}
Golumbic, M.C., Ries, B.: On the intersection graphs of orthogonal line
  segments in the plane: characterizations of some subclasses of chordal
  graphs. Graphs and Combinatorics  \textbf{29}(3),  499--517 (2013)

\bibitem{gonccalves2020not}
Gon{\c{c}}alves, D.: Not all planar graphs are in {PURE-4-DIR}. Journal of
  Graph Algorithms and Applications  \textbf{24}(3),  293--301 (2020)

\bibitem{gonccalves2018planar}
Gon{\c{c}}alves, D., Isenmann, L., Pennarun, C.: Planar graphs as
  {L}-intersection or {L}-contact graphs. In: Proceedings of the Twenty-Ninth
  Annual ACM-SIAM Symposium on Discrete Algorithms. pp. 172--184. SIAM (2018)

\bibitem{hartman1991grid}
Hartman, I.B.A., Newman, I., Ziv, R.: On grid intersection graphs. Discrete
  Mathematics  \textbf{87}(1),  41--52 (1991)

\bibitem{jain2022b_0}
Jain, S., Pallathumadam, S.K., Rajendraprasad, D.: {B$_0$-VPG} {R}epresentation
  of {AT}-free {O}uterplanar {G}raphs. In: Conference on Algorithms and
  Discrete Applied Mathematics. pp. 103--114. Springer (2022)

\bibitem{kant1996augmenting}
Kant, G.: Augmenting {O}uterplanar {G}raphs. Journal of Algorithms
  \textbf{21}(1),  1--25 (1996)

\bibitem{kratochvil1991string}
Kratochv{\'\i}l, J.: String graphs. {II}. {R}ecognizing string graphs is
  {NP}-hard. Journal of Combinatorial Theory, Series B  \textbf{52}(1),  67--78
  (1991)

\bibitem{kratochvil1994intersection}
Kratochv{\'\i}l, J., Matou\v{s}ek, J.: Intersection graphs of segments. Journal
  of Combinatorial Theory, Series B  \textbf{62}(2),  289--315 (1994)

\bibitem{mehrabi2017approximation}
Mehrabi, S.: {A}pproximation {A}lgorithms for {I}ndependence and {D}omination
  on {B$_1$-VPG} and {B$_1$-EPG} {G}raphs. arXiv preprint arXiv:1702.05633
  (2017)

\bibitem{pallathumadam2022characterization}
Pallathumadam, S.K., Rajendraprasad, D.: Characterization of {B}$_0$-{VPG}
  {C}ocomparability {G}raphs and a {2D} {V}isualization of their {P}osets.
  Order pp. 1--20 (2022)

\bibitem{schaefer2003recognizing}
Schaefer, M., Sedgwick, E., {\v{S}}tefankovi{\v{c}}, D.: Recognizing string
  graphs in {NP}. Journal of Computer and System Sciences  \textbf{67}(2),
  365--380 (2003)

\bibitem{scheinerman1984intersection}
Scheinerman, E.R.: Intersection classes and multiple intersection parameters of
  graphs. Princeton University (1984)

\bibitem{west1991open}
West, D.: Open problems. SIAM J. Discrete Math. Newslett.  \textbf{2},  10--12
  (1991)

\end{thebibliography}

\end{document}